\documentclass[12pt]{amsart}
\usepackage{latexsym,
 amssymb,
 graphicx,
 enumerate, 
 paralist,
 tabmac,
 fullpage,
 hyperref,anyfontsize
}
\usepackage[dvipsnames]{xcolor}

\graphicspath{{figures/}}

\theoremstyle{plain}
   \newtheorem{theorem}{Theorem}[section]
   \newtheorem{proposition}[theorem]{Proposition}
   \newtheorem{prop}[theorem]{Proposition}
   \newtheorem{lemma}[theorem]{Lemma}
   \newtheorem{corollary}[theorem]{Corollary}
   \newtheorem{conjecture}[theorem]{Conjecture}
   
   \newtheorem*{theorem*}{Theorem}
   \newtheorem{defprop}[theorem]{Definition--Proposition}
\theoremstyle{definition}
   \newtheorem{definition}[theorem]{Definition}
   \newtheorem{example}[theorem]{Example}
   
   \newtheorem{question}[theorem]{Question}
   \newtheorem{remark}[theorem]{Remark}

\numberwithin{equation}{section}

\newcommand\qbin[3]{\left[\genfrac{}{}{0pt}{}{#1}{#2}\right]_{#3}}

\newcommand\Symm{\mathfrak{S}}

\newcommand\Sym{\operatorname{Sym}}

\newcommand\wt{\operatorname{wt}}

\newcommand{\std}{\operatorname{std}}

\newcommand{\Des}{\operatorname{Des}}
\newcommand\affS{\widetilde{\Symm}}

\newcommand{\ol}[1]{\overline{#1}}
\newcommand{\Green}{\mathcal{Q}}

\newcommand{\ambcs}{\overset{\textrm{AMBC}}{\longleftrightarrow}}

\newcommand{\AAA}{{\mathcal{A}}}
\newcommand{\DDD}{{\mathcal{D}}}

\newcommand\ZZ{{\mathbb{Z}}}
\newcommand\NN{{\mathbb{N}}}

\newcommand{\RRSS}{\operatorname{RRSS}}

\newcommand{\La}{\Lambda}
\newcommand{\la}{\lambda}

\newcommand{\tw}[1]{\widetilde{#1}}
\newcommand{\wh}[1]{\widehat{#1}}
\newcommand{\ov}[1]{\overline{#1}}
\newcommand{\rsks}{\overset{\textrm{RSK}}{\longleftrightarrow}}
\newcommand{\SSYT}{{\rm{SSYT}}}
\newcommand{\SYT}{{\rm{SYT}}}
\newcommand{\RC}{{\rm{RC}}}
\newcommand{\CLR}{{\rm{CLR}}}

\newcommand{\trLR}{{\rm{tr_{LR}}}}

\newcommand{\sort}{{\rm{sort}}}
\usepackage{cancel}
\DeclareMathOperator{\pr}{pr}
\newcommand{\ds}{\displaystyle}

\begin{document}

\title{An affine generalization of evacuation}
\author{M.~Chmutov, G.~Frieden, D.~Kim, J.B.~Lewis, and E.~Yudovina}
\begin{abstract}
We establish the existence of an involution on tabloids that is analogous to Sch\"utzenberger's evacuation map on standard Young tableaux.  We find that the number of its fixed points is given by evaluating a certain Green's polynomial at $q = -1$, and satisfies a ``domino-like'' recurrence relation.  
\end{abstract}

\maketitle

\tableofcontents

\section{Introduction}

This paper concerns an analogue in the affine symmetric group of the operation called \emph{evacuation} (or \emph{Sch\"utzenberger's involution}) in the finite symmetric group $\Symm_n$, and the associated beautiful combinatorial and representation-theoretic story.  We begin by summarizing a few highlights of this story.  Evacuation is an involution on standard Young tableaux of a given shape.  Under the Robinson--Schensted bijection, it corresponds to a natural ``rotation'' involution on the symmetric group of permutations.  It commutes with Knuth relations and so gives a symmetry of the \emph{dual equivalence graphs}, whose vertices are tableaux and whose edges are Knuth moves. 
In terms of representation theory, this means that evacuation permutes Kazhdan--Lusztig cells of symmetric groups. As it preserves the shape of each standard Young tableau, it stabilizes each two-sided cell and permutes the left and right cells contained in such two-sided cells.

Evacuation can be computed by many combinatorial algorithms, none of which is completely straightforward.  However, for certain nice shapes (notably, rectangles), it can be described in very simple terms.  Evacuation also has interesting enumerative properties: its fixed points are counted by an instance of Stembridge's \emph{$q = -1$ phenomenon} and are in bijection with \emph{domino tableaux} of the same shape.  
In turn, the domino tableaux are closely related to Kazhdan--Lusztig cells of Weyl groups and Springer theory for types B, C, and D.
This story is recalled in more detail in Sections~\ref{sec:finite} and~\ref{sec: q=-1}. 

Our project is to construct a parallel story for the \emph{affine symmetric group} $\affS_n$.  This group is an infinite analogue of the symmetric group, and much of the beautiful combinatorics and representation theory of the symmetric group can be extended to the affine setting.  We give its formal definition in Section~\ref{sec:affine}. The analogue of the Robinson--Schensted correspondence for $\affS_n$ is the \emph{affine matrix ball construction} (AMBC) described by Chmutov--Pylyavskyy--Yudovina \cite{CPY}, based on the work of Shi \cite{Shi_GRS}, associating to each affine permutation two \emph{tabloids} and some additional data. Our first main theorem (Theorem~\ref{evacuation exists theorem}) establishes the existence of an affine analogue of evacuation, in the following sense: there is a natural ``rotation'' involution on $\affS_n$ which corresponds, via AMBC, to an involution on tabloids.  In Section~\ref{sec:defining evacuation}, we show that this map corresponds to the usual evacuation map when restricted to tabloids that happen to be tableaux; that is plays well with Knuth moves and dual equivalence graphs; and that it has a particularly simple form when computed on tabloids of rectangular shape.  We also give several algorithms by which it may be computed in general.

In Section~\ref{sec:self-evacuating}, we consider the fixed points of the affine evacuation map.  Our second main result (Theorem~\ref{main counting theorem}) establishes that the fixed points of affine evacuation are counted by an evaluation of a Green's polynomial at $q = -1$, and (using results of the third-named author) that they satisfy a recurrence with a domino flavor. One of the key steps in the proof is an evaluation of the Kostka--Foulkes polynomials at $q=-1$ (Theorem~\ref{thm: KF at -1}); this result is established using the theory of rigged configurations. We also give elementary combinatorial proofs of Theorem~\ref{main counting theorem} for certain simple shapes.

In Section~\ref{sec:open problems}, we give a number of open problems and additional remarks.

The main thrust of the present paper is combinatorial.  For more discussion of the representation-theoretic aspects of this work, see the companion paper \cite{KimInvolution} by the third-named author.

\subsection*{Acknowledgements}
The authors are grateful to Kevin Dilks, for initially suggesting the idea of generalizing evacuation in this context; to Sam Hopkins and Thomas McConville, who informed us of our parallel work on related questions; to Brendon Rhoades, for helpful conversations about Kostka--Foulkes polynomials; to Pavlo Pylyavskyy, for numerous conversations; and to Vic Reiner, for many fruitful questions and suggestions. MC was supported in part by NSF grant DMS-1503119; GF was supported in part by NSF grants DMS-1464693 and DMS-0943832; JBL was supported in part by NSF grant DMS-1401792.

\section{Background}
\subsection{Finite symmetric group}
\label{sec:finite}

We begin by describing the story that we hope to emulate in the affine setting.  For further background, we recommend \cite[Ch.~7, Appendix 1]{EC2}, \cite[Ch.~3]{Sagan}, \cite[\S 1.5, 2.3, A3 and Chs.~5--6]{BB}, and \cite{Manivel}.

\subsubsection{Finite permutations}
Among the many ways to represent permutations in the symmetric group $\Symm_n$, one may write them 
\begin{compactitem}
\item in \emph{one-line notation}, as words containing each element of $[n] := \{1, \ldots, n\}$ exactly once;
\item as $n \times n$ \emph{permutation matrices} (an example is shown in Figure \ref{fig:fin perm}); or 
\item as the elements of the \emph{Coxeter group} of type $A_{n - 1}$, having generators $\{s_1, \ldots, s_{n - 1}\}$ and relations $s_i^2 = 1$, $s_is_js_i = s_js_is_j$ if $i - j = \pm1$, and $s_is_j = s_js_i$ otherwise.
\end{compactitem} 
\begin{figure}
\[
\resizebox{.15\textwidth}{!}{\input{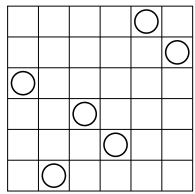}}
\] 
	\caption{
The permutation matrix of the permutation $561342\in\Symm_6$, where the balls stand for entries equal to $1$ while the rest of the entries are $0$.  In Coxeter generators, it can be expressed as $s_4s_5s_3s_4s_2s_3s_4s_5s_1s_2$.}
\label{fig:fin perm}
\end{figure}
As a permutation, $s_i$ is the \emph{simple transposition} that exchanges $i$ and $i + 1$.  It is convenient to represent the relations by the associated \emph{Dynkin diagram}:
\[
\resizebox{.3\textwidth}{!}{\begin{picture}(0,0)%
\includegraphics{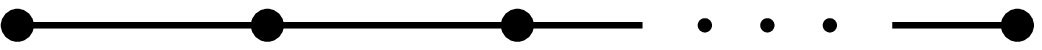}%
\end{picture}%
\setlength{\unitlength}{3947sp}%
\begingroup\makeatletter\ifx\SetFigFont\undefined%
\gdef\SetFigFont#1#2#3#4#5{%
  \reset@font\fontsize{#1}{#2pt}%
  \fontfamily{#3}\fontseries{#4}\fontshape{#5}%
  \selectfont}%
\fi\endgroup%
\begin{picture}(4966,447)(3518,-3275)
\put(8401,-3211){\makebox(0,0)[lb]{\smash{{\SetFigFont{20}{24.0}{\rmdefault}{\mddefault}{\updefault}{\color[rgb]{0,0,0}$s_{n-1}$}%
}}}}
\put(3601,-3211){\makebox(0,0)[lb]{\smash{{\SetFigFont{20}{24.0}{\rmdefault}{\mddefault}{\updefault}{\color[rgb]{0,0,0}$s_1$}%
}}}}
\put(4801,-3211){\makebox(0,0)[lb]{\smash{{\SetFigFont{20}{24.0}{\rmdefault}{\mddefault}{\updefault}{\color[rgb]{0,0,0}$s_2$}%
}}}}
\put(6001,-3211){\makebox(0,0)[lb]{\smash{{\SetFigFont{20}{24.0}{\rmdefault}{\mddefault}{\updefault}{\color[rgb]{0,0,0}$s_3$}%
}}}}
\end{picture}%
}
\] 
The \emph{longest element} $w_0$ in $\Symm_n$ is the element whose shortest words in the $s_i$ are of maximum length among all permutations in $\Symm_n$; it has one-line notation $n(n - 1) \cdots 1$, and its permutation matrix has all $1$s on the main anti-diagonal.  The longest element has multiplicative order $2$; conjugating by it is a natural involution on $\Symm_n$ that may be realized as
\begin{compactitem}
\item sending the word $w_1 \cdots w_n$ to its ``reverse-complement" $(n + 1 - w_n) \cdots (n + 1 - w_1)$;
\item rotating the permutation matrix by $180$ degrees around its center; or
\item substituting $s_i \leftrightarrow s_{n - i}$ in any expression for a permutation as a product of generators. 
\end{compactitem}
This last characterization corresponds to the unique nontrivial automorphism of the Dynkin diagram of $\Symm_n$.

A \emph{descent} in a permutation $w = w_1\cdots w_n$ is a position $i \in [n - 1]$ such that $w_i > w_{i + 1}$.  An \emph{inverse descent} is a descent of the inverse permutation $w^{-1}$; equivalently, it is a value $i \in [n - 1]$ such that the value $i + 1$ appears to its left in the one-line notation for $w$.  In Coxeter language, descents are \emph{right descents} while inverse descents are \emph{left descents}.  The set of descents of $w$ is denoted $\Des(w)$.

Given a permutation $w$, to apply a \emph{Knuth move} is to switch $w_i$ and $w_{i + 1}$ for some $i$, provided that at least one of $w_{i - 1}$ and $w_{i + 2}$ has value between that of $w_i$ and $w_{i + 1}$.  Equivalently, there is a Knuth move from $w$ to $w'$ if the two permutations differ by a simple transposition and have descent sets incomparable under inclusion.  If there is a Knuth move that changes $w$ to $w'$, then there is also a Knuth move (involving the same collection of positions) changing $w'$ to $w$.  Thus there is an (undirected) graph $G_n$ on $\Symm_n$ whose edges are the Knuth moves.

It is easy to see from the characterizations above that $w$ has a descent in position $i$ if and only if $w_0 w w_0$ has a descent in position $n - i$, and similarly that $w$ is connected to $w'$ by a Knuth move if and only if $w_0 w w_0$ is connected to $w_0 w' w_0$ by a Knuth move.  Thus conjugation by $w_0$ is an automorphism of the graph $G_n$.

\subsubsection{Tableaux and Robinson--Schensted}
\label{sec:tableaux and RS}
A \emph{partition} $\lambda = \langle \lambda_1, \ldots, \lambda_k \rangle$ of $n$ is a finite, weakly decreasing sequence of positive integers such that $\lambda_1 + \ldots + \lambda_k = n$.  To indicate that $\lambda$ is a partition of $n$, we write $\lambda \vdash n$, or $|\lambda| = n$. We write $\ell(\lambda)$ for the number of parts of $\lambda$.  A partition may be represented visually by its \emph{Young diagram}, a left-aligned array of boxes having $\lambda_i$ boxes in the $i$th row from the top.  A \emph{standard Young tableau} of \emph{shape} $\lambda$ is a filling of the Young diagram of $\lambda$ with the numbers $1, \ldots, |\lambda|$, each used once, such that numbers increase down columns and across rows.

The \emph{Robinson--Schensted correspondence} (abbreviated \emph{RS} in this paper) is a combinatorial bijection between $\Symm_n$ and pairs $(P, Q)$ of standard Young tableaux of the same shape $\lambda$, a partition of $n$.  It plays a central role in both the combinatorics and representation theory of the symmetric group.  If $w \overset{\mathrm{RS}}{\longleftrightarrow} (P, Q)$, we may sometimes denote $P = P(w)$ and $Q = Q(w)$, and if $P$ and $Q$ have shape $\lambda$ then we also say that $w$ has shape $\lambda$.  We may also refer to $P$ and $Q$ as the \emph{insertion tableau} and \emph{recording tableau} of $w$, respectively.

A \emph{descent} in a standard Young tableau $T$ is a number $i \in [n - 1]$ such that $i + 1$ lies in a lower row in $T$ than $i$, and the set of descents of $T$ is denoted $\Des(T)$.  Descent sets are preserved by RS: if the permutation $w$ corresponds to the pair $(P, Q)$ under RS, then $\Des(w) = \Des(Q)$ and $\Des(w^{-1}) = \Des(P)$.

A \emph{Knuth move} on a tableau $T$ is the operation that switches entries $i$ and $i + 1$, provided that the result is a standard Young tableau and that its descent set is incomparable under inclusion to $\Des(T)$.  Two permutations $w$ and $w'$ are connected by a Knuth move if and only if $P(w) = P(w')$ and $Q(w)$ is connected by a Knuth move to $Q(w')$. 
  Following \cite{assaf_degI}, we call the graph  $\DDD_\lambda$ on tableaux of shape $\lambda$ whose edges are Knuth moves the \emph{dual equivalence graph} of shape $\lambda$.  In fact, RS is a graph isomorphism between the dual equivalence graph $\DDD_\lambda$ and the (induced) subgraph of $G_n$ whose vertices are the permutations with fixed insertion tableau $P$ of shape $\lambda$.

\subsubsection{Evacuation}
While RS does not interact well with the group operation of $\Symm_n$ in general, in the case of conjugation by $w_0$ its behavior is well-understood: there is an involution $e$ on the set of standard Young tableaux of any fixed shape $\lambda$ such that if $w$ in $\Symm_n$ corresponds to $(P, Q)$ under RS, then $w_0 w w_0$ corresponds to $(e(P), e(Q))$.  This operation, named \emph{evacuation} by Sch\"utzenberger, may be computed on a tableau $T$ by:

\begin{itemize}
\item rotating $T$ by 180 degrees, replacing each entry $i$ with $n+1-i$, and restoring the resulting tableau to its original shape using Sch\"utzenberger's \emph{jeu de taquin};
\item using repeated applications of Sch\"utzenberger's \emph{promotion} operation to successively ``evacuate'' the entries of $T$;
\item applying a particular sequence of \emph{Bender--Knuth involutions} to $T$ (see \cite{KirBer}); or
\item a growth diagram, as in \cite[\S A1.2]{EC2}.  
\end{itemize}

It follows from the discussion above, concerning the interaction of descents with RS and with conjugation by $w_0$, that the integer $i$ is a descent for the tableau $T$ if and only if $n - i$ is a descent for $e(T)$.  Consequently, tableaux $T$ and $T'$ are connected by a Knuth move if and only if $e(T)$ and $e(T')$ are, and so evacuation is an automorphism of the graph $\DDD_\lambda$.

\subsection{Affine symmetric group}
\label{sec:affine}

In this section, we describe the affine analogues of permutations, tableaux, and RS.

\subsubsection{Affine permutations}
Abstractly, the \emph{affine symmetric group} $\affS_n$ is the Coxeter group of affine type $\widetilde{A}_{n - 1}$, having generators $\{s_0, s_1, \ldots, s_{n - 1}\}$ and relations $s_i^2 = 1$, $s_is_js_i = s_js_is_j$ if $i - j \equiv\pm1 \pmod{n}$, and $s_is_j = s_js_i$ otherwise.  The corresponding Dynkin diagram is
\[
\resizebox{.3\textwidth}{!}{\begin{picture}(0,0)%
\includegraphics{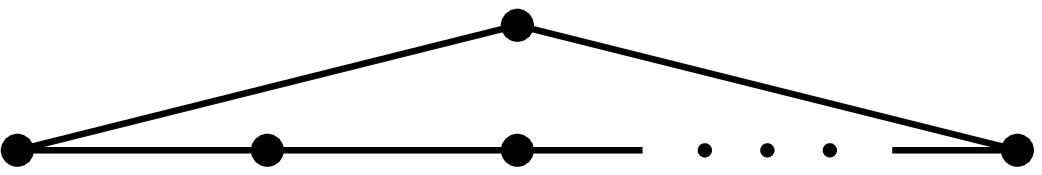}%
\end{picture}%
\setlength{\unitlength}{3947sp}%
\begingroup\makeatletter\ifx\SetFigFont\undefined%
\gdef\SetFigFont#1#2#3#4#5{%
  \reset@font\fontsize{#1}{#2pt}%
  \fontfamily{#3}\fontseries{#4}\fontshape{#5}%
  \selectfont}%
\fi\endgroup%
\begin{picture}(4966,1047)(3518,-3275)
\put(6001,-2611){\makebox(0,0)[lb]{\smash{{\SetFigFont{20}{24.0}{\rmdefault}{\mddefault}{\updefault}{\color[rgb]{0,0,0}$s_0$}%
}}}}
\put(3601,-3211){\makebox(0,0)[lb]{\smash{{\SetFigFont{20}{24.0}{\rmdefault}{\mddefault}{\updefault}{\color[rgb]{0,0,0}$s_1$}%
}}}}
\put(4801,-3211){\makebox(0,0)[lb]{\smash{{\SetFigFont{20}{24.0}{\rmdefault}{\mddefault}{\updefault}{\color[rgb]{0,0,0}$s_2$}%
}}}}
\put(6001,-3211){\makebox(0,0)[lb]{\smash{{\SetFigFont{20}{24.0}{\rmdefault}{\mddefault}{\updefault}{\color[rgb]{0,0,0}$s_3$}%
}}}}
\put(8401,-3211){\makebox(0,0)[lb]{\smash{{\SetFigFont{20}{24.0}{\rmdefault}{\mddefault}{\updefault}{\color[rgb]{0,0,0}$s_{n-1}$}%
}}}}
\end{picture}%
}
\] 
The group elements may be represented as certain periodic bijections of the integers \cite{Lusztig, Erikssons}: we say that a bijection $w: \ZZ \to \ZZ$ is an \emph{affine permutation}\footnote{
Technically, the combinatorial objects defined here are actually the \emph{extended} affine permutations; the affine permutations have the extra condition $\sum_{i=1}^n w(i) = \sum_{i=1}^n i = \frac{n(n+1)}{2}$.  This distinction is not important in what follows: some remarks explicitly invoke the Coxeter structure and so make sense only in the non-extended group, but all the results and proofs are valid regardless which setting one chooses to work in.}
 if $w(i + n) = w(i) + n$ for all $i \in \ZZ$. We will typically write $w_i$ in place of $w(i)$.  As is the case for the finite symmetric group, such permutations may be represented in different ways (see Figure \ref{fig:aff perm}):
\begin{itemize}
\item in \emph{window notation}, as words $[w_1, w_2, \ldots, w_n]$ of length $n$ containing one representative of each equivalence class of integers modulo $n$; or
\item as infinite periodic permutation matrices, having rows and columns indexed by $\ZZ$, one nonzero entry in each row and column, and periodicity under translation by $(n, n)$.
\end{itemize}
For the matrix representation, we use matrix coordinates, so row numbers increase from top to bottom, column numbers increase from left to right, and $(1, 2)$ represents the first row and second column.

\begin{figure}
\[
\resizebox{.4\textwidth}{!}{\input{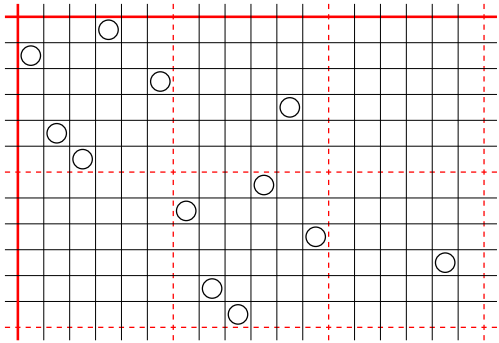}}
\] 
	\caption{The matrix of the affine permutation $[4,1,6,11,2,3]\in\affS_6$. The column to the right of the bold red line is column 1, and analogously for rows.}
\label{fig:aff perm}
\end{figure}

Because $\affS_n$ is infinite, there is no longest element by which to conjugate.  However, there is a natural group automorphism $r: \affS_n \to \affS_n$ that may be characterized in several equivalent ways:
\begin{itemize}
\item in terms of the Coxeter diagram, it is the extension of the diagram automorphism interchanging the simple generators
$
s_i \longleftrightarrow s_{n - i}
$
for $i \in [n - 1]$ and fixing $s_0$;
\item in terms of the permutation matrix, it is the rotation by $180$ degrees that preserves the square $[n] \times [n]$; and
\item in terms of the window notation, it is \emph{reverse-complement}: write the word from right to left, then subtract each entry from $n + 1$, so that $[w_1, \ldots,  w_n]$ becomes $[n + 1 - w_n, \ldots, n + 1 - w_1]$.
\end{itemize}

\begin{remark}
The first bullet point only makes sense for the affine (non-extended) permutations, while the second and third make sense in the full extended group.  

While the operation $r$ does not correspond to conjugation by any element inside $\affS_n$, if we embed this group inside the (much larger) group $\Sym(\ZZ)$ of all permutations of the integers then $r$ is conjugation by the permutation $c^{(n)}$ defined by $c^{(n)}(i) = n + 1 - i$ for all $i \in \ZZ$.
\end{remark}

Let $[\ol{n}]$ be the set of equivalence classes of integers modulo $n$, and for an integer $i$ let $\ol{i}$ denote the class containing $i$.  A \emph{descent} in an affine permutation $w$ is an equivalence class $\ol{i} \in [\ol{n}]$ of positions such that $w_i > w_{i + 1}$.  An \emph{inverse descent} of $w$ is a descent in the inverse affine permutation $w^{-1}$.  As before, in terms of the Coxeter group these are the \emph{right} and \emph{left} descents.    The \emph{descent set} $\Des(w)$ is the set of descents of $w$.

Two affine permutations $w$ and $w'$ are related by a \emph{Knuth move} if $w'$ differs from $w$ by exchanging the two elements $w_i$ and $w_{i + 1}$ (as well as all the other pairs of elements implied by the periodicity condition), provided that $\Des(w)$ and $\Des(w')$ are incomparable.  Equivalently, this requires that at least one of $w_{i - 1}$ and $w_{i + 2}$ lies numerically between $w_i$ and $w_{i + 1}$.  Thus, there is a graph $\tw{G}_n$ whose vertices are the affine permutations, with two permutations connected by an edge if and only if they are related by a Knuth move.  In the language of Kazhdan--Lusztig theory, this graph consists of the bidirected edges in the \emph{$W$-graph} for the affine symmetric group; the induced subgraphs on its connected components are what Stembridge \cite{StembridgeMolecules} calls the ``Kazhdan--Lusztig molecules.''

\subsubsection{Tabloids and AMBC}
\label{sec:AMBC}

Given a partition $\lambda \vdash n$, a \emph{tabloid} of shape $\lambda$ is an equivalence class of fillings of the Young diagram of $\lambda$ with $[\ol{n}]$, where two fillings are considered equivalent if they differ only in the arrangement of elements within rows.  Thus, tabloids are equinumerous with \emph{row-strict tableaux} filled bijectively with $[n]$, in which entries are required to increase along rows (but with no column condition), and with standard skew Young tableaux of shape $\left\langle n, n - \lambda_1, n - \lambda_1 - \lambda_2, \ldots\right\rangle / \left\langle n - \lambda_1, n - \lambda_1 - \lambda_2, \ldots \right\rangle$.  A \emph{descent} in a tabloid $T$ is an element $\ol{i} \in [\ol{n}]$ such that $\ol{i + 1}$ lies in a lower row in $T$ than $\ol{i}$, and the set of descents of $T$ is denoted $\Des(T)$.  Two tabloids $T, T'$ are related by a \emph{Knuth move} if $T'$ differs from $T$ by exchanging two elements $\ol{i}$ and $\ol{i + 1}$, provided that $\Des(T)$ and $\Des(T')$ are incomparable.  The \emph{(Kazhdan--Lusztig) dual equivalence graph} $\AAA_\lambda$ is the graph on tabloids of shape $\lambda$, with two tabloids connected by an edge if and only if they differ by a Knuth move.

The analogue of RS for $\affS_n$ is the \emph{affine matrix ball construction} (\emph{AMBC}), which sends each affine permutation to a triple $(P, Q, \rho)$ where $P$ and $Q$ are \emph{tabloids}
of the same shape $\lambda \vdash n$, and the \emph{weight vector} $\rho$ is an integer vector of length $\ell(\lambda)$ satisfying certain inequalities.  The reverse map $\Psi: (P, Q, \rho) \mapsto w \in \affS_n$ is defined on all triples $(P, Q, \rho)$ consisting of two tabloids of the same shape and an integer weight vector of the correct length; when restricted to weight vectors satisfying the appropriate inequalities, it is the inverse of AMBC. In particular, if $w = \Psi(P,Q,\rho)$, then we have $w \ambcs (P,Q,\rho')$ for some weight vector $\rho'$. (In this paper we mostly ignore the weight vector; but see Question~\ref{weight question} below.) 

As the next result shows, AMBC respects descents in the same way as RS.
\begin{prop}[{\cite[Prop.~3.6]{CLP}}]
\label{prop 3.6}
If $w \ambcs (P, Q, \rho)$ then $\Des(w) = \Des(Q)$ and $\Des(w^{-1}) = \Des(P)$.
\end{prop}  
As a consequence of this proposition, AMBC also respects Knuth moves.  In fact, more is true.
\begin{prop}[{\cite[Thm.~3.11, Lem.~3.22, and Prop.~3.23]{CLP}}]
\label{covering theorem}
If $w$ is an affine permutation and $w'$ differs from $w$ by a Knuth move, then $P(w') = P(w)$ and $Q(w')$ differs from $Q(w)$ by a Knuth move.  When restricted to a single connected component in the graph $\tw{G}_n$, the map $w \mapsto Q(w)$ is a graph covering of $\AAA_\lambda$.  Moreover, if tabloids $Q$ and $Q'$ are related by a Knuth move, then for any tabloid $P$ and weight $\rho$, the permutation $w = \Psi(P, Q, \rho)$ is related by a Knuth move to a permutation $w'$ with $Q(w') = Q'$.
\end{prop}

\section{Affine evacuation}
\label{sec:defining evacuation}

In this section, we first establish the existence of an involution $e$ on tabloids of a given shape, analogous to evacuation on tableaux.  We show that it coincides with the usual evacuation on the tabloid analogues of standard tableaux (validating the choice of terminology), and that it has a simple description on tabloids of rectangular shape.  Finally, we give three alternate characterizations of the map $e$: in terms of an operation called the \emph{combinatorial $R$-matrix}; by sending tabloids through the Robinson--Schensted--Knuth correspondence; and as an ``asymptotic" (\emph{\`a la} Pak \cite{pak}) version of the usual evacuation map.  The first two reformulations will be used in \S\ref{sec: counting fixed points} to study the fixed points of the map $e$.

\subsection{Evacuation exists}
The main result of this section is that the rotation operation $r$ defined in \S\ref{sec:affine} interacts with AMBC in the way one would hope.
\begin{theorem}
\label{evacuation exists theorem}
There is an involution $e$ on the set of tabloids of shape $\lambda$ such that if $w \ambcs (P, Q, \rho)$ then $r(w) \ambcs (e(P), e(Q), \rho')$ for some weight $\rho'$.
\end{theorem}

The main idea of the proof is to explicitly construct $e$ as an automorphism of the dual equivalence graph $\AAA_\lambda$ by showing that $r$ interacts in predictable ways with descents and Knuth moves.  The proof is spread over the next several results, and the statement of the theorem is subsumed by Theorem~\ref{thm: def1 for evacuation} below.

\begin{prop}
\label{rotation-shape}
If $w$ has shape $\lambda$ then $r(w)$ has shape $\lambda$.
\end{prop}
\begin{proof}
Given an affine permutation $w$, the nonzero entries (\emph{balls}) in its permutation matrix belong to $n$ equivalence classes under translation by $(n, n)$.  Following \cite[\S3.2]{CPY}, define the \emph{Shi poset} $P_w$ of $w$ as follows: the vertices of $P_w$ are equivalence classes of balls, and one equivalence class $C_1$ is less than another equivalence class $C_2$ if there exist balls $b_1 \in C_1$ and $b_2 \in C_2$ such that $b_2$ is north-east of $b_1$.  (Because of the periodicity of the matrix, it is equivalent to ask that every $b_1 \in C_1$ has some $b_2 \in C_2$ that is north-east of it, or that every $b_2 \in C_2$ has some $b_1 \in C_1$ that is south-west of it.)

The \emph{Greene--Kleitman shape} \cite{Greene, GreeneKleitman} of a finite poset $P$ is the partition $\langle \lambda_1, \lambda_2, \ldots\rangle$ such that $\lambda_1$ is the size of the largest antichain of $P$, $\lambda_1 + \lambda_2$ is the size of the largest union of two antichains of $P$, and so on.  It is easy to see that the Greene--Kleitman shape of any poset is equal to the Greene--Kleitman shape of its dual.


By \cite[Thm.~9.4]{CPY}, using the description of Shi's algorithm given in \cite[\S9]{CPY}, we know that the shape of $w$ is equal to the Greene--Kleitman shape of $P_w$. It is straightforward from the definition that the Shi poset of $r(w)$ is the dual of the Shi poset of $w$. Thus the shapes of $w$ and of $r(w)$ are both equal to the Greene--Kleitman shape of $P_w$, and so are equal to each other.
\end{proof}

\begin{prop}
\label{rotation-descent-set}
If $w$ has descent set $I = \{\ol{i_1}, \ldots, \ol{i_k}\} \subset [\overline{n}]$ and inverse descent set $J = \{\ol{j_1}, \ldots, \ol{j_\ell}\} \subset [\overline{n}]$ then $r(w)$ has descent set $n - I = \{\ol{n - i_1}, \ldots, \ol{n - i_k}\}$ and inverse descent set $n - J = \{\ol{n - j_1}, \ldots, \ol{n - j_\ell}\}$.
\end{prop}
\begin{proof}
This is immediate from any of the definitions of $r$.
\end{proof}

\begin{definition}
If a Knuth move connects two affine permutations $w$ and $w'$, and $\ol{i}$ is contained in one of $\Des(w)$ and $\Des(w')$ while $\ol{i + 1}$ is contained in the other, then we say that the Knuth move has \emph{type} $\ol{i}$.  We extend this definition verbatim to Knuth moves on tabloids. 
\end{definition}

A Knuth move may have either one or two types: for example, the Knuth move connecting $[0, 2, 4]$ to $[0, 4, 2]$ has unique type $\ol{2}$, whereas the Knuth move connecting $[1, 5, 0]$ and $[1, 0, 5]$ has types $\ol{1}$ and $\ol{2}$.  Note that a Knuth move of type $\ol{i}$ is realized by switching $\ol{i}$ and $\ol{i+1}$, or $\ol{i+1}$ and $\ol{i+2}$.
Thus, when a Knuth move has two types, they are adjacent in the cyclic order on $[\ol{n}]$.
\begin{prop}
\label{prop:exists unique Knuth move}
Suppose $T$ is a tabloid and exactly one of $\ol{i}$ and $\ol{i+1}$ belongs to $\Des(T)$. Then there exists a unique tabloid $T'$ connected to $T$ by a Knuth move of type $\ol{i}$.
\end{prop}
(The corresponding statement for permutations is \cite[Prop.~3.8]{CLP}.)
\begin{proof}
If a single Knuth move affects both descents, it must either exchange the entries $\ol{i}$ and $\ol{i+1}$ or the entries $\ol{i+1}$ and $\ol{i+2}$ in $T$. By symmetry, we may suppose $\ol{i}\in \Des(T)$ and $\ol{i+1}\notin\Des(T)$. So $\ol{i+1}$ is in a row strictly lower than $\ol{i}$ while $\ol{i+2}$ is in a row weakly higher than $\ol{i+1}$.  We consider two cases.

If $\ol{i+2}$ is in a row strictly lower than $\ol{i}$ in $T$, then exchanging $\ol{i}$ and $\ol{i+1}$ is the desired Knuth move of type $\ol{i}$.  In this case, exchanging $\ol{i+1}$ and $\ol{i+2}$, even if it is a Knuth move, does not remove $\ol{i}$ from the descent set and hence is not a Knuth move of type $\ol{i}$.

If, on the other hand, $\ol{i+2}$ is weakly above $\ol{i}$ in $T$, then exchanging $\ol{i+1}$ and $\ol{i+2}$ is the desired Knuth move, while exchanging $\ol{i}$ and $\ol{i+1}$ does not add $\ol{i+1}$ to the descent set.
\end{proof}

Next we describe how Knuth moves interact with the map $r$.

\begin{prop}
\label{prop:rotation preserves knuth moves}
Suppose that $w$ and $w'$ differ by a Knuth move of type $\ol{i}$.  Then $r(w)$ and $r(w')$ differ by a Knuth move of type $\ol{n - i - 1}$.
\end{prop}
\begin{proof}
This follows immediately from the definition of type and Proposition~\ref{rotation-descent-set}.
\end{proof}

For $\ol{k} \in [\overline{n}]$, there is a unique tabloid of shape $\lambda$ having singleton descent set~$\{\ol{k}\}$: its first row is filled with the residue classes $\ol{k}, \ol{k - 1}, \ldots, \ol{k - \lambda_1 + 1}$, its second row is filled with $\ol{k - \lambda_1}, \ol{k - \lambda_1 - 1}, \ldots, \ol{k - \lambda_1 - \lambda_2 + 1}$, and so on. We call such tabloids \emph{reverse row superstandard} and we denote the one with descent set $\{\ol{k} \}$ by $\RRSS(\lambda, \ol{k})$.

\begin{prop}
\label{prop: evac on RRSS}
If $w \ambcs (P, \RRSS(\lambda, \ol{k}), \rho)$ then $r(w) \ambcs (P', \RRSS(\lambda, \ol{n - k}), \rho')$ for some tabloid $P'$ and some weight $\rho'$.	
\end{prop}
\begin{proof}
Suppose $w \ambcs (P, \RRSS(\lambda, \ol{k}), \rho)$. By Proposition~\ref{prop 3.6}, we have that $\Des(w) = \{\ol{k}\}$. It follows by Proposition~\ref{rotation-descent-set} that $\Des(r(w)) = \{\ol{n - k}\}$, and so (again by Proposition~\ref{prop 3.6}) that $\Des(Q(r(w))) = \{\ol{n - k}\}$.  By Proposition~\ref{rotation-shape}, $Q(r(w))$ is of shape $\lambda$.  Thus $Q(r(w))$ is equal to $\RRSS(\lambda, \ol{n - k})$, the unique tabloid of shape $\lambda$ having descent set $\{\ol{n - k}\}$.
\end{proof}

We now come to the main result of this section.

\begin{theorem}
\label{thm: def1 for evacuation}
There exists a unique automorphism $e$ of the graph $\AAA_\lambda$ such that
\begin{itemize}
\item $e\left(\RRSS(\lambda, \ol{k})\right) = \RRSS(\lambda, \ol{n - k})$ for every $\ol{k} \in [\ol{n}]$, and
\item if the edge $\{T_1, T_2\}$ is a Knuth move of type $\ol{k}$ then the edge $\{e(T_1), e(T_2)\}$ is a Knuth move of type $\ol{n - k - 1}$.
\end{itemize}
Moreover, $e$ is an involution on the set of tabloids of shape $\lambda$, and if $w \ambcs (P, Q, \rho)$ then $r(w) \ambcs (e(P), e(Q), \rho')$ for some weight $\rho'$.
\end{theorem}
\begin{proof}
For a tabloid $P$ of shape $\lambda$ and a vector $\rho\in\ZZ^{\ell(\lambda)}$, define a map $e_{P,\rho}: \AAA_\lambda\to \AAA_\lambda$ as follows. Given a tabloid $T$, let $w_{P,\rho} = \Psi(P, T, \rho)$ (where $\Psi$ is the inverse map of AMBC, mentioned above in \S\ref{sec:AMBC}); define $e_{P,\rho}(T) = Q(r(w_{P,\rho}))$. We will show that $e_{P,\rho}$ is actually independent of $P$ and $\rho$, and hence will be able to drop them from the notation.

We proceed by induction on the distance (in $\AAA_\lambda$) from $T$ to a reverse row superstandard tabloid (which is finite by \cite[Lem.~7.4]{CLP}). If $T = \RRSS(\lambda, \ol{k})$, then by Proposition \ref{prop: evac on RRSS} we have $e_{P,\rho}(T) = \RRSS(\lambda, \ol{n-k})$, and this manifestly does not depend on $P$ or $\rho$. For the inductive step, suppose $T$ is connected by a Knuth move of type $\ol{i}$ to a tabloid $T'$ that satisfies the condition that $e_{P,\rho}(T')$ is independent of $P$ and $\rho$. By Proposition~\ref{covering theorem}, $w:=\Psi(P,T,\rho)$ is connected by a Knuth move 
to $w'=\Psi(P, T', \rho')$ for some $\rho'$, and by Proposition~\ref{prop 3.6}
this move is of type $\ol{i}$. By Proposition~\ref{prop:rotation preserves knuth moves}, $r(w)$ is connected to $r(w')$ by a Knuth move of type $\ol{n-i-1}$. Again by Propositions~\ref{covering theorem} and~\ref{prop 3.6}, $e_{P,\rho'}(T')$ and $e_{P,\rho}(T)$ are connected by a Knuth move of type $\ol{n-i-1}$. But $e_{P,\rho'}(T')$ does not depend on $P$ or $\rho'$, and $e_{P,\rho}(T)$ is the unique (by Proposition~\ref{prop:exists unique Knuth move}) tabloid connected to $e_{P,\rho'}(T')$ by a Knuth move of type $\ol{n-i-1}$; thus it also does not depend on $P$ or $\rho$. By induction, $e_{P,\rho}$ indeed does not depend on $P$ or $\rho$, so from now on we drop them from the notation. 

The same argument shows that any automorphism $e'$ of $\AAA_\lambda$ satisfying the conditions of the theorem is, in fact, equal to $e$: the two automorphisms agree on the reverse row superstandard tabloids by hypothesis, and the inductive argument shows that they also agree on every other tabloid.

By the construction of $e$, we have for any $w$ that $Q(r(w)) = e(Q(w))$. Thus, since $r$ is an involution, $e$ must also be an involution.  By taking inverses (using \cite[Prop.~3.1]{CLP} and the fact that $r(w^{-1}) = r(w)^{-1}$), it follows that $P(r(w)) = e(P(w))$. This finishes the proof.
\end{proof}

\subsection{Computing evacuation of rectangles and standard tableaux}
\label{sec:special cases}

In this section, we use the characterization in Theorem~\ref{thm: def1 for evacuation} to show that the affine evacuation map $e$ has nice behaviors in two cases: when the tabloid $T$ is essentially a standard Young tableau, and when the tabloid $T$ is of rectangular shape. 

Given a standard Young tableau $T$, one may naturally associate to it a tabloid $\tw{T}$: if the entries of row $i$ of $T$ are $a_1, \ldots, a_k$, then the entries of row $i$ of $\tw{T}$ are $\ol{a_1}, \ldots, \ol{a_k}$.  If a tabloid $\tw{T}$ is the associated tabloid of a tableau $T$, we say that $\tw{T}$ is \emph{standardizable}, and that $T$ is its \emph{standardization}.
\begin{prop}
If $T$ is a standard Young tableau and $\widetilde{T}$ is the associated tabloid, then $e(\widetilde{T})$ is a standardizable tabloid whose associated tableau is the evacuation of $T$.
\end{prop}
This result validates our repetition of terminology and notation in the affine and finite cases.
\begin{proof}
By \cite[Thm.~10.2]{CPY} we have the following result: if $w \in \Symm_n$ corresponds to the pair $(P, Q)$ of standard Young tableaux under RS, then (viewing $w$ as an element of $\affS_n$ in the natural way) it corresponds to the triple $(\tw{P}, \tw{Q}, 0)$ under AMBC.  Moreover, it is easy to see that in this case $w_0 w w_0$ and $r(w)$ coincide as elements of $\affS_n$.  The result follows immediately by taking $w$ to be any permutation with $P(w) = T$.
\end{proof}

In the finite case, it is typically not possible to compute the evacuation of a tableau just ``by eye''.  One important exception is the case of tableaux whose shape is a rectangle, in which case one simply rotates the tableau by 180 degrees, and replaces each entry $i$ with $n+1-i$, where $n$ is the size of the rectangle.\footnote{The origin of this observation is murky: according to Stanley \cite[p.~12]{StanleyEvac}, it follows easily from Sch\"utzenberger's initial study \cite{Schutzenberger} of evacuation, but it is not clear where it is first written down.}  The corresponding result for tabloids of rectangular shape is straightforward.
\begin{prop}
\label{rectangles}
Suppose that $T$ is a tabloid of shape $\langle a^b \rangle$.  
For $i = 1, \ldots, b$, denote the entries in row $i$ of $T$ by $\{\ol{T_{i, 1}}, \ldots, \ol{T_{i, a}}\}$.  Then $e(T)$ is the tabloid whose $i$th row contains entries $\{\ol{ab + 1 - T_{b + 1 - i, 1}}, \ldots, \ol{ab + 1 - T_{b + 1 - i, a}}\}$ for all $i$. That is, $e(T)$ is obtained from $T$ by turning it upside down and ``reflecting'' the entries according to the rule $\ol{j}\mapsto\ol{ab+1-j}$. 
\end{prop}
\begin{proof}
It is easy to check that the result is correct for reverse row superstandard tabloids, and that the operation interacts with Knuth moves in the appropriate way. 
\end{proof}

\subsection{Computing affine evacuation using the combinatorial $R$-matrix}
\label{sec:R-matrix}

In this section, we realize affine evacuation in terms of an algorithm called the \emph{combinatorial $R$-matrix}.  In this context, we think of a tabloid as a tuple of its rows, and the combinatorial $R$-matrix gives a natural way to re-order parts of a tabloid into non-partition shapes.  We begin with a general definition of the combinatorial $R$-matrix, starting in the two-row case. 

A \emph{semistandard Young tableau} of shape $\lambda$ is a filling of the Young diagram of $\lambda$ so that the result is strictly increasing down columns and weakly increasing across rows.  Let $B^k$ denote the set of semistandard Young tableaux of shape $\langle k \rangle$ with entries in $[n]$.
\begin{definition}
\label{defn: comb R}
The \emph{combinatorial $R$-matrix} is the map
\[
R : B^{k_2} \times B^{k_1} \rightarrow B^{k_1} \times B^{k_2}
\]
that sends $(a,b) \mapsto (a',b')$, according to the following algorithm:
\begin{enumerate}
\item Write the multiset of entries in $a$ and $b$ horizontally, in increasing order. Below each entry of $a$ (respectively, $b$), place a right (resp., left) parenthesis. For a given number, the right parentheses should occur before the left parentheses.
\item Say that a left parenthesis and a right parenthesis are a \emph{matched pair} if the right parenthesis occurs after the left parenthesis, and there are no other parentheses between them. Recursively remove matched pairs until none remain.  At this point, there remain $\alpha$ right parentheses followed by $\beta$ left parentheses for some nonnegative integers $\alpha, \beta$.
\item Replace these unmatched symbols with $\beta$ right parentheses, followed by $\alpha$ left parentheses.
\item Add back the removed parentheses, and let $a'$ (respectively, $b'$) be the multiset of numbers above a right (resp., left) parenthesis.
\end{enumerate}
\end{definition}

\begin{example}
Let
\[
a = \tableau[sY]{2,3,4,5,5,5,7} \in B^7, \quad\quad\quad b = \tableau[sY]{1,1,1,2,4,5,5,6,6} \in B^9.
\]
Step 1 of the above algorithm leads to
\[
\begin{array}{cccccccccccccccc}
1 & 1 & 1 & 2 & 2 & 3 & 4 & 4 & 5 & 5 & 5 & 5 & 5 & 6 & 6 & 7 \\
( & ( & ( & ) & ( & ) & ) & ( & ) & ) & ) & ( & ( & ( & ( & )
\end{array}.
\]
After recursively removing matched pairs of parentheses, we are left with
\[
\begin{array}{cccccccccccccccc}
1 & 1 & 1 & 2 & 2 & 3 & 4 & 4 & 5 & 5 & 5 & 5 & 5 & 6 & 6 & 7 \\
&&&&&&&&&& ) & ( & ( & ( &&
\end{array}.
\]
Step 3 changes these unmatched parentheses into
\[
\begin{array}{cccccccccccccccc}
1 & 1 & 1 & 2 & 2 & 3 & 4 & 4 & 5 & 5 & 5 & 5 & 5 & 6 & 6 & 7 \\
&&&&&&&&&& ) & ) & ) & ( &&
\end{array},
\]
so two 5's are ``transferred'' by the combinatorial $R$-matrix to produce the pair $(a', b')$:
\[
a' = \tableau[sY]{2,3,4,5,5,5,5,5,7} \in B^9, \quad\quad\quad b' = \tableau[sY]{1,1,1,2,4,6,6} \in B^7.
\]
\end{example}

\begin{remark}
The combinatorial $R$-matrix of Definition \ref{defn: comb R} is a special case of an isomorphism coming from affine crystal theory. The algorithm for $R$ given above is due to Nakayashiki and Yamada \cite[Rule 3.11]{NakayashikiYamada}.\footnote{Our parenthesis rule is easily seen to be equivalent to the diagrammatic rule in \cite{NakayashikiYamada}; note, however, that we have reversed the order of the factors $a$ and $b$.}
\end{remark}
\begin{remark}
Shimozono showed that the combinatorial $R$-matrix can also be described in terms of the plactic monoid (as originally defined by Lascoux and Sch\"utzenberger \cite{LascouxSchutzenberger}; see also \cite[Ch.~2]{Fulton}).  For semistandard tableaux $a$ and $b$, let $a * b$ denote the product in the plactic monoid; this product can be computed, for example, by jeu de taquin or Schensted insertion. For $(a,b) \in B^{k_2} \times B^{k_1}$, the map $R(a,b) = (a',b')$ is characterized by the property that $(a',b')$ is the unique pair in $B^{k_1} \times B^{k_2}$ such that $a * b = a' * b'$ (see \cite[Ex.~4.10]{ShimozonoDummies}).
\end{remark}

Next, we extend this definition to more complicated shapes.  Let $k_1, \ldots, k_d$ be a sequence of positive integers, and set $B^{k_d, \ldots, k_1} = B^{k_d} \times \cdots \times B^{k_1}$. (Arranging the subscripts in decreasing order turns out to be notationally convenient when using the English notation for Young diagrams.) For $i = 1, \ldots, d-1$, define
\[
R_i : B^{k_d, \ldots, k_{i+1}, k_i, \ldots, k_1} \rightarrow B^{k_d, \ldots, k_i, k_{i+1}, \ldots, k_1}
\]
to be the map that applies the combinatorial $R$-matrix to the factors $B^{k_{i+1}} \times B^{k_i}$, and leaves the other factors alone.  The next result shows that these operations generate an action of the symmetric group.

\begin{proposition}[{\cite[Prop.~4.7 and Thm.~4.8]{ShimozonoDummies}}]
\label{prop:R braid}
The maps $R_i$ satisfy the relations $R_i^2 = \mathrm{id}$, $R_i R_j = R_j R_i$ if $|i-j| > 1$, and $R_i R_{i+1} R_i = R_{i+1} R_i R_{i+1}$.
Thus, for any permutation $\sigma \in \Symm_d$, there is a well-defined map $R_\sigma : B^{k_d, \ldots, k_1} \rightarrow B^{k_{\sigma(d)}, \ldots, k_{\sigma(1)}}$.  Moreover, if $(k_d, \ldots, k_1) = (k_{\sigma(d)}, \ldots, k_{\sigma(1)})$, then $R_\sigma$ is the identity.
\end{proposition}

Given $b = (b_d, \ldots, b_1) \in B^{k_d, \ldots, k_1}$, define the \emph{content} (or \emph{weight}) of $b$ to be the tuple
$
\wt(b) = (a_1, \ldots, a_n)
$
such that $a_i$ is the total number of times that $i$ appears in $b_d, \ldots, b_1$. Clearly, $R_\sigma$ preserves content.

A \emph{(strict) composition} of $n$ is a sequence $\mu = \langle \mu_1, \ldots, \mu_d \rangle$ of positive integers such that $n = \mu_1 + \cdots + \mu_d$. We extend the definition of Young diagram from partitions to compositions in the natural way. Consequently, we may speak of the set $\mathcal{T}(\mu)$ of tabloids of shape $\mu$ for any composition $\mu$, and all the key definitions (descent, Knuth move, etc.) carry over from \S \ref{sec:AMBC}. We identify $\mathcal{T}(\mu)$ with the set of elements of $B^{\mu_d, \ldots, \mu_1}$ of content $\langle 1^n \rangle$ by treating each row of a tabloid as a one-row tableau with entries in $[n]$. For each permutation $\sigma$, this identification induces a map $R_\sigma : \mathcal{T}(\mu) \rightarrow \mathcal{T}(\sigma(\mu))$, which we also call the combinatorial $R$-matrix.

The \emph{promotion map}\footnote{
The operation of promotion was originally defined by Sch\"utzenberger \cite{SchutzenbergerPromotion} for arbitrary finite posets, and is often seen in the context of Young tableaux (e.g., in \cite[\S 2.8]{Manivel}, where the corresponding map is called $\partial$).  However, the promotion of a standardizable tabloid typically does \emph{not} agree with the promotion of its standardization.  In particular, the poset corresponding to a partition $\lambda$ in the finite setting is the diagram of $\lambda$, viewed as a subset of $\NN \times \NN$ with the product order, while in the affine setting it is the disjoint union of chains of length $\lambda_1, \lambda_2, \ldots$. On the other hand, the two notions of promotion do agree if we view a tabloid as a standard skew tableau, as mentioned in \S\ref{sec:AMBC}.}
$\pr: \mathcal{T}(\mu) \to \mathcal{T}(\mu)$ acts by adding $1$ to each entry of a tabloid $T$.  More generally, one may define $\pr: B^{\mu_d, \ldots, \mu_1} \to B^{\mu_d, \ldots, \mu_1}$ to be the component-wise action that replaces a one-row tableau of content $(a_1, \ldots, a_n)$ with the unique one-row tableau of content $(a_n, a_1, \ldots, a_{n-1})$.

\begin{lemma}
\label{lem: comb R commutes with promotion}
The combinatorial $R$-matrix commutes with promotion.
\end{lemma}

\begin{proof}
Step 3 of the algorithm in Definition~\ref{defn: comb R} may be thought of as repeatedly matching the right-most unpaired left parenthesis with the left-most unpaired right parenthesis until the parentheses remaining unmatched are all of the same type, and then changing the remaining parentheses to the opposite type.  Equivalently, the parentheses are matched as if they lie on a closed loop, rather than a line.  This version of the algorithm clearly respects the cyclic action of promotion.
\end{proof}

\begin{lemma}
\label{lem: R preserves descents}
The combinatorial $R$-matrix preserves the descent set of a tabloid.
\end{lemma}

\begin{proof}
Since each $R_\sigma$ is a composition of maps that only affect two consecutive rows, it suffices to consider the case of two-row tabloids. Furthermore, since $R_1$ is an involution, it suffices to show that $\Des(T) \subset \Des(R_1(T))$. Suppose $T \in \mathcal{T}(\langle \mu_1, \mu_2 \rangle)$, and let $(a,b)$ be the corresponding element of $B^{\mu_2, \mu_1}$. Suppose $\ol{i}$ is a descent of $T$. By Lemma~\ref{lem: comb R commutes with promotion}, we may assume that $i \neq n$, so that $i$ is an entry of $b$ and $i+1$ is an entry of $a$. When one applies the algorithm of Definition~\ref{defn: comb R}, a left parenthesis is placed below $i$ and a right parenthesis below $i+1$, so these parentheses form a matched pair. This means $i \in b'$ and $i+1 \in a'$, so $\ol{i}$ is a descent of $R_1(T)$.
\end{proof}

The next result is the main result of this section; it gives an algorithm to compute affine evacuation in terms of the combinatorial $R$-matrix.

\begin{theorem}
\label{thm: alg for e}
Given a tabloid $T \in \mathcal{T}(\mu)$, the affine evacuation $e(T)$ may be computed as follows:
\begin{enumerate}
\item reverse the order of the rows of $T$, thus obtaining a tabloid in $\mathcal{T}(w_0(\mu))$;
\item replace each entry $\ol{i}$ with $\ol{n+1-i}$;
\item apply $R_{w_0}$ to restore the tabloid to its original shape.
\end{enumerate}
\end{theorem}

\begin{proof}
Let $\gamma(T)$ denote the result of applying the first two steps of this procedure to $T$, and let $\wh{e}(T) := R_{w_0}(\gamma(T))$ denote the final result. By Theorem \ref{thm: def1 for evacuation}, it suffices to show that $\wh{e}(\RRSS(\mu, \ol{k})) = \RRSS(\mu, \ol{n-k})$ for each $\ol{k} \in [\ol{n}]$, and that if $T_1, T_2$ are connected by a Knuth move of type $\ol{k}$, then $\wh{e}(T_1), \wh{e}(T_2)$ are connected by a Knuth move of type $\ol{n-k-1}$.

By construction, $\gamma\left(\RRSS(\mu, \ol{k})\right) = \RRSS(w_0(\mu), \ol{n-k})$. Since the $R$-matrix permutes shapes, we know that $R_{w_0}(\RRSS(w_0(\mu), \ol{n-k}))$ has shape $\mu$.  By Lemma \ref{lem: R preserves descents}, we know that $\Des(R_{w_0}(\RRSS(w_0(\mu), \ol{n-k}))) = \{\ol{n-k}\}$. So $R_{w_0}(\RRSS(w_0(\mu), \ol{n-k}) = \RRSS(\mu, \ol{n-k})$, the unique tabloid of shape $\mu$ with this descent set. Thus $\wh{e}(\RRSS(\mu, \ol{k})) = \RRSS(\mu, \ol{n-k})$, as needed.

It is clear that $T_1$ and $T_2$ are connected by a Knuth move of type $\ol{k}$ if and only if $\gamma(T_1)$ and $\gamma(T_2)$ are connected by a Knuth move of type $\ol{n-k-1}$, so it remains to show that if $T_1$ and $T_2$ are connected by a Knuth move of type $\ol{k}$, then so are $R_i(T_1)$ and $R_i(T_2)$ for all $i$.

Suppose that $T_1$ and $T_2$ are connected by a Knuth move of type $\ol{k}$. Let $U_1 = R_i(T_1)$ and $U_2 = R_i(T_2)$ for some $i$. By Lemma \ref{lem: R preserves descents}, we know that $U_j$ has the same descent set as $T_j$, so it suffices to show that $U_1$ and $U_2$ are related by swapping $\ol{k+1}$ with $\ol{k}$ or $\ol{k+2}$. Using Lemma \ref{lem: comb R commutes with promotion}, we may assume $k \in [n-2]$, so we drop the bars above the numbers. Suppose that $T_2$ is obtained from $T_1$ by swapping $k+1$ and $k+2$ (the case where $T_1$ and $T_2$ are related by swapping $k+1$ and $k$ is dealt with similarly). If the union of rows $i$ and $i+1$ in $T_1$ contains at most one of $k+1$ and $k+2$, then it's clear from the procedure of Definition \ref{defn: comb R} that the map $R_i$ commutes with the swap.

Now suppose that $k+1$ and $k+2$ are both in the union of rows $i$ and $i+1$ of $T_1$. By symmetry, we may assume that the descent set of $T_1$ contains $\ol{k}$ but not $\ol{k+1}$, and the descent set of $T_2$ contains $\ol{k+1}$ but not $\ol{k}$. This requires that in $T_1$, the numbers $k$ and $k+2$ are in row $i$, and $k+1$ is in row $i+1$. When we compute $U_1$, $k$ and $k+2$ get left parentheses and $k+1$ gets a right parenthesis. The parentheses under $k$ and $k+1$ form a matched pair, and we are left with a left parenthesis under $k+2$. Thus, $k$ and $k+1$ remain in their respective rows, and the fate of $k+2$ depends on the arrangement of the other elements in rows $i$ and $i+1$. When we compute $U_2$, the parentheses under $k+1$ and $k+2$ are matched, and we are left with a left parenthesis under $k$; thus, $k+1$ and $k+2$ remain in their respective rows, and the fate of $k$ depends on the arrangement of the other elements in rows $i$ and $i+1$ in precisely the same way that the fate of $k+2$ depends on the other elements when computing $U_1$. If $k+2$ ends up in row $i$ of $U_1$ (and thus $k$ ends up in row $i$ of $U_2$), then $U_1$ and $U_2$ differ by swapping $k+1$ and $k+2$. If $k+2$ ends up in row $i+1$ of $U_1$, then $U_1$ and $U_2$ differ by swapping $k+1$ and $k$. This completes the proof.
\end{proof}

\subsection{Affine evacuation and RSK}
\label{sec: evac and RSK}

In this section, we give another procedure for computing affine evacuation, which will play a crucial role in the study of its fixed points in \S\ref{sec: counting fixed points}. It is based on the Robinson--Schensted--Knuth correspondence, and we assume familiarity with that map as in \cite[Ch.~7]{EC2}.

For a partition $\la$, let $\SYT(\la)$ denote the set of standard Young tableaux of shape $\la$, and let $\SSYT_n(\la)$ denote the set of semistandard Young tableaux of shape $\la$ with entries in $[n]$. If $\mu = \langle \mu_1, \ldots, \mu_d \rangle$ is a strict composition of $n$, let $\SSYT(\la, \mu)$ be the set of semistandard Young tableaux of shape $\la$ and content $\mu$. The Robinson--Schensted--Knuth (RSK) correspondence gives a bijection
\[
\ds B^{\mu_d, \ldots, \mu_1} \rsks \bigsqcup_{\la \, \vdash \, n} \SSYT_n(\la) \times \SSYT(\la, w_0(\mu))
\]
by sending $(b_d, \ldots, b_1) \in B^{\mu_d, \ldots, \mu_1}$ to the two-row array $
\begin{pmatrix} 
1^{\mu_d} & 2^{\mu_{d - 1}} & \cdots & d^{\mu_1} \\
b_d  & b_{d - 1} & \cdots & b_1
\end{pmatrix}
$ and then using the Schensted insertion algorithm (as in \cite[\S7.11]{EC2}) to produce the pair $(P, Q)$.  Identifying $\mathcal{T}(\mu)$ with the elements of $B^{\mu_d, \ldots, \mu_1}$ of weight $\langle 1^n \rangle$, RSK restricts to a bijection
\[
\ds \mathcal{T}(\mu) \rsks \bigsqcup_{\la \, \vdash \, n} \SYT(\la) \times \SSYT(\la, w_0(\mu)).
\]
\begin{example}
\label{ex:rsk on a tabloid}
Consider the tabloid 
\[
\tableau[sY]{\ol{2}, \ol{3}, \ol{5}, \ol{7} \\ \ol{1}, \ol{4} \\ \ol{6}}.
\]
We first convert it to the two-row array $\begin{pmatrix} 1 & 2 & 2 & 3 & 3 & 3 & 3\\ 6 & 1 & 4 & 2 & 3 & 5 & 7 \end{pmatrix}$ by reading the rows from the bottom up, and then apply RSK to produce the pair
\[
(P,Q) = \left( \, \tableau[sY]{1, 2, 3, 5, 7 \\ 4 \\ 6} \; , \quad \tableau[sY]{1, 2, 3, 3, 3 \\ 2 \\3 } \, \right).
\]
\end{example}

To describe what affine evacuation on $\mathcal{T}(\mu)$ corresponds to on the other side of this bijection, we need several additional definitions.  The first of these is an action of the symmetric group on semistandard tableaux that was first defined by Lascoux and Sch\"utzenberger.

\begin{defprop}[{Lascoux--Sch\"utzenberger \cite[\S4]{LascouxSchutzenberger}}]
\label{si action}
For $i \geq 1$, let $s_i$ act on a semistandard tableau $U$ by the following algorithm:
\begin{enumerate}
\item Concatenate the rows of $U$, beginning with the last row, to form the \emph{reading word} $W$.
\item Place a right (respesctively, left) parenthesis beneath each occurrence of the letter $i$ (resp., $i+1$), and recursively remove pairs of matched parentheses as in Step 2 of the algorithm in Definition \ref{defn: comb R}. 
\item After removing all matched pairs, there remain $\alpha$ right parentheses followed by $\beta$ left parentheses, which correspond to a subword $i^\alpha (i+1)^\beta$ in $W$. Replace this subword with $i^\beta (i+1)^\alpha$. 
\end{enumerate}
The resulting word is the reading word of a unique semistandard tableau $s_i(U)$ of shape $\la$.  Moreover, these maps satisfy the braid relations (as in Proposition~\ref{prop:R braid}) and so extend to an action of the symmetric group $\Symm_d$ on the set of semistandard tableaux of shape $\lambda$ with entries in $[d]$. Given $U \in \SSYT_d(\la)$ and $w \in \Symm_d$, we write $w(U)$ for the action of $w$ on $U$. Moreover, if $U$ has content $\mu$ and $w(\mu) = \mu$, then $w(U) = U$.
\end{defprop}

By definition, this action permutes the content, i.e., if $U$ has content $\nu$ then $w(U)$ has content $w(\nu)$.  The connection between this action and the combinatorial $R$-matrix is given in the following result.

\begin{proposition}[{\cite[Prop.~5.1]{ShimozonoDummies}}]
\label{R matrix through RSK}
If $(b_d, \ldots, b_1) \rsks (P,Q)$, then $R_i(b_d, \ldots, b_1) \rsks (P, s_{d-i}(Q))$.\footnote{The result in \cite{ShimozonoDummies} has $s_i$ rather than $s_{d-i}$. This is because the definition of the recording tableau $Q$ in that paper is the ``reverse'' of the definition in \cite{EC2} that we use.}
\end{proposition}

Next, we describe a version of the Sch\"utzenberger evacuation map that acts on $\SSYT_d(\la)$.  We denote it by $e_d$ to emphasize the dependence on $d$. 

\begin{definition}
\label{def: e_d}
Given a tableau $U$ in $\SSYT_d(\la)$, let $W = W_1 \cdots W_n \in [d]^n$ be the reading word of $U$ (as in Definition--Proposition~\ref{si action}), and let $W' = (d + 1 - W_n) \cdots (d + 1 - W_1)$ be the reverse-complement of $W$.  Under RSK, $W'$ corresponds to a pair $(P(W'), Q(W'))$, and we define $e_d(U) = P(W')$.  In particular, if $U$ is a \emph{standard} tableau, then $W$ is a permutation, $W' = w_0 W w_0$ as elements of $\Symm_n$, and $e_n(U) = e(U)$ is the usual evacuation defined in \S\ref{sec:finite}.

Further, we define a map $e_d^* = w_0 \circ e_d$, where $w_0$ is the longest element of $\Symm_d$ acting as in Definition--Proposition~\ref{si action}.
\end{definition}

If $U$ has content $\nu$, then $e_d(U)$ has content $w_0(\nu)$, where $w_0$ is the longest permutation in $\Symm_d$. Thus, $e_d^*$ is content-preserving. Furthermore, the maps $e_d$ and $w_0$ commute,\footnote{This follows from the identity $e_d \circ s_i = s_{d-i} \circ e_d$, which is proved in, e.g., \cite[Prop. 2.87(iii)]{ShimozonoDummies}.} so $e_d^*$ is an involution on $\SSYT(\la, \mu)$.

The next result is the main result of this subsection, expressing the action of the affine evacuation map $e$ on a tabloid $T$ in terms of its image under RSK.

\begin{prop}
\label{prop: alt alg for e}
Suppose $\mu$ is a composition of $n$ with $d$ parts, and $T \in \mathcal{T}(\mu)$ is a tabloid of shape $\mu$. If $T \rsks (P,Q)$, then
\[
e(T) \rsks (e_n(P), e_d^*(Q)).
\]
\end{prop}

\begin{proof}
Fix a tabloid $T$ of shape $\mu$, and let $(P, Q)$ be the image of $T$ under RSK.  By definition, $P$ is the insertion tableau of the reading word $W$ of $T$.  
Let $\gamma(T) \in \mathcal{T}(w_0(\mu))$ be the tabloid obtained from $T$ by reversing the order of the rows and replacing each element $\ol{i}$ with $\ol{n-i+1}$.  Suppose $\gamma(T) \rsks (P', Q')$, so that $P'$ is the insertion tableau of the row word $W'$ of $\gamma(T)$.  By definition of $\gamma$, $W'$ is the reverse-complement of $W$, and so we have $P' = e_n(P)$ by Definition~\ref{def: e_d}. Similarly, the symmetry of RSK (see, e.g., \cite[\S7.13]{EC2}) shows that $Q' = e_d(Q)$.
By Theorem \ref{thm: alg for e}, we have $e(T) = R_{w_0}(\gamma(T))$, and the result follows from Proposition~\ref{R matrix through RSK}.
\end{proof}

\begin{example}
We continue with Example~\ref{ex:rsk on a tabloid}.  Since $P$ is a standard Young tableau, 
\[
e_7(P) = e(P) = \tableau[sY]{1, 2, 4, 6, 7 \\ 3 \\ 5}.
\]
We compute $e_3^*(Q)$ as follows.  First, the row word of $Q$ is $3212333$, with reverse-complement $1112321$.  The insertion tableau for this word is $e_3(Q) = \tableau[sY]{1,1,1,1,2 \\ 2 \\3}$.  Next, we apply $w_0 = s_1 \cdot s_2 \cdot s_1$:
\[
e_3(Q) = \tableau[sY]{1,1,1,1,2 \\ 2 \\3}
\overset{s_1}{\longrightarrow}
\tableau[sY]{1,1,2,2,2 \\ 2 \\3}
\overset{s_2}{\longrightarrow}
\tableau[sY]{1,1,3,3,3 \\ 2 \\3}
\overset{s_1}{\longrightarrow}
\tableau[sY]{1,2,3,3,3 \\ 2 \\3}
= e^*_3(Q).
\]
(In this case, by coincidence, $e^*_3(Q) = Q$.)  The pair $(e_7(P), e_3^*(Q))
$ corresponds under RSK to the two-row array $\begin{pmatrix} 1 & 2 & 2 & 3 & 3 & 3 & 3 \\ 5 & 1 & 3 & 2 & 4 & 6 & 7 \end{pmatrix}$, and this array encodes the tabloid
\[
e(T) = 
\tableau[sY]{\ol{2}, \ol{4}, \ol{6}, \ol{7} \\ \ol{1}, \ol{3} \\ \ol{5}}.
\]
\end{example}

\subsection{Asymptotic definition of affine evacuation}
\label{sec:asymptotic}

In this section, we describe an ``asymptotic" realization of the evacuation of a tabloid, in the spirit of \cite{pak}.  To this end, we describe several new pieces of terminology, to be used only in this section.  Given a sequence $(T_0, T_1, \ldots )$ of tableaux, each filled bijectively with some subset of the integers, say that the sequence \emph{stabilizes} if, for every integer $i$, there exists an integer $j = j(i)$ such that for all sufficiently large $k$, one has that $i$ appears in row $j$ of $T_k$.  If a sequence $(T_0, \ldots)$ stabilizes, we say moreover that it is \emph{periodic} if $j(i) = j(i + n)$ for all $i$.  In this case, there is a tabloid $T$ such that $\ol{i}$ appears in row $j(i)$ of $T$; we say this tabloid is the \emph{limit} of the sequence.

Given a tabloid $T$, one may associate a sequence of tableaux as follows: let $w$ be any affine permutation such that $P(w) = T$, let $w^{(i)}$ be the finite sub-word $w^{(i)} = (w_{-in + 1}, w_{-in + 2}, \ldots,$ $w_{in + n - 1}, w_{in + n})$ of the doubly-infinite sequence of values of $w$, and let $T_i = P(w^{(i)})$ be the insertion tableau for this word under RS.  By \cite[Thm.~7.3]{CPY}, the original tabloid $T$ is the limit of this sequence.
\begin{prop}
\label{prop:asymptotic}
For any tabloid $T$, with the sequence $(T_0, T_1, \ldots)$ of tableaux as in the preceding paragraph, one has that the sequence $(e(T_0), e(T_1), \ldots )$ is periodic, with limit $e(T)$.  (Here by $e(T_k)$ we mean that one should apply the usual (finite) evacuation to $T_k$, under the order-isomorphism between the set of entries of $T_k$ and $\{1, 2, \ldots, |T_k|\}$.)
\end{prop}

\begin{example}
Suppose $T = \tableau[sY]{\ol{2}, \ol{4} \\ \ol{1} \\ \ol{3}}$.  We may choose $w = [1, 2, 0, 7]$, so that $w^{(0)} = (1, 2, 0, 7)$, $w^{(1)} = (-3, -2, -4, 3, 1, 2, 0, 7, 5, 6, 4, 11)$, 
and so on.  The insertion tableaux $T_0, T_1, T_2,\ldots$ for these words are 
\def\Tscale{1.3}
\[
\tableau[Y]{0, 2, 7 \\ 1} \; , \quad 
\tableau[Y]{-4, -2, 0, 2, 4, 6, 11 \\ -3, 1, 5 \\ 3, 7} \; , \quad
\tableau[Y]{-8, -6, -4, -2, 0, 2, 4, 6, 8, 10, 15 \\ -7, -3, 1, 5, 9 \\ -1, 3, 7, 11} \; , 
\quad \ldots \; ,
\]
with evacuations $e(T_0), e(T_1), e(T_2), \ldots$ being
\[
\tableau[Y]{0, 1, 2 \\ 7} \; , \quad
\tableau[Y]{-4, -3, 0, 2, 4, 6, 7 \\ -2, 3, 11 \\ 1, 5} \; , \quad
\tableau[Y]{-8, -7, -4, -2, 0, 2, 4, 6, 8, 10, 11 \\ -6, -1, 3, 7, 15 \\ -3, 1, 5, 9} \; , 
\quad \ldots.
\]
The latter sequence stabilizes and is periodic, with $i$ in row $j$ of $T_k$ for $k \gg 0$ if and only if $\ol{i}$ is in row $j$ of 
\[
e(T) = \tableau[sY]{\ol{2}, \ol{4} \\ \ol{3} \\ \ol{1}}.
\]
\end{example}

\begin{proof}[Proof sketch]
It follows from \cite[Thm.~7.3]{CPY} that the sequence $(T_0, T_1, \ldots)$ is periodic, and moreover that the row in which $i$ appears in the limit is the same as the row in which $\ol{i}$ appears in $T$.  Except for a bounded number of the smallest and largest values (corresponding to the fact that $\{w_{-in + 1}, \ldots, w_{in + n}\}$ is not necessarily an interval of integers), $e(T_k)$ coincides with $P(r(w^{(k)}))$.  But again by \cite[Thm.~7.3]{CPY}, the sequence $\left(P(r(w^{(k)}))\right)_{k \geq 0}$ is periodic, and for each integer $i$ the row in which $i$ appears in the limit is the same as the row in which $\ol{i}$ appears in $P(r(w)) = e(T)$.
\end{proof}

\begin{remark}
There is no particular reason to restrict the sequence of words of $w$ to include nested windows, as above: with more technical details, one could take any sequence of sequences $[w_{a_i}, w_{a_i + 1}, \ldots, w_{b_i}]$ with $a_i \to -\infty$ and $b_i \to \infty$ and get the same result.  For further possible generalizations, see \S\ref{sec:asymptotic remark}.
\end{remark}


\section{Enumeration of self-evacuating tabloids}
\label{sec:self-evacuating}
\label{sec: counting fixed points}

Given any group action on a set, it is natural to ask about the fixed points of the action.  In the case of evacuation, this is to ask about the \emph{self-evacuating} tableaux and tabloids.  In the finite case, these fixed points have a fascinating enumeration, and we begin this section by recounting it.  The rest of the section is devoted to the statement and proof of Theorem~\ref{main counting theorem}, our affine analogue, which gives the enumeration of self-evacuating tabloids of a given shape.

\subsection{Self-evacuating tableaux, domino tableaux, and $q = -1$}
\label{sec: q=-1}
When $|\lambda|$ is even, a \emph{domino tableau} of shape $\lambda$ is a division of the cells of the Young diagram of $\lambda$ into pairs of adjacent cells (``dominoes") that are numbered by the integers $1, \ldots, |\lambda|/2$ in such a way that for each $k$, the union of the dominoes numbered $1, \ldots, k$ is again a Young diagram.  For example, the three domino tableaux of shape $\langle 4, 2\rangle$ are
\[
\resizebox{.5\textwidth}{!}{\begin{picture}(0,0)%
\includegraphics{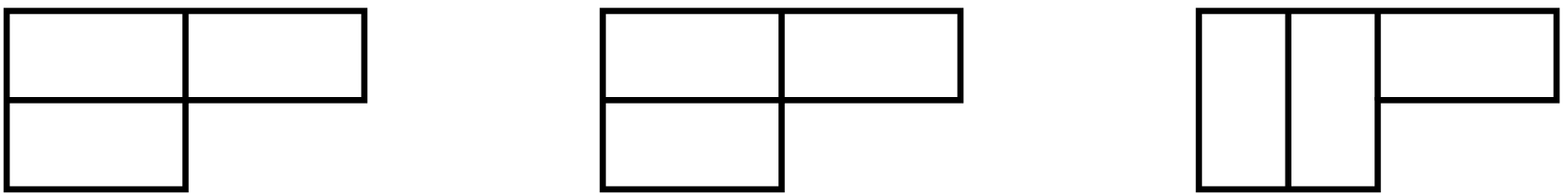}%
\end{picture}%
\setlength{\unitlength}{3947sp}%
\begingroup\makeatletter\ifx\SetFigFont\undefined%
\gdef\SetFigFont#1#2#3#4#5{%
  \reset@font\fontsize{#1}{#2pt}%
  \fontfamily{#3}\fontseries{#4}\fontshape{#5}%
  \selectfont}%
\fi\endgroup%
\begin{picture}(8090,966)(2968,-3844)
\put(10951,-3511){\makebox(0,0)[lb]{\smash{{\SetFigFont{29}{34.8}{\rmdefault}{\mddefault}{\updefault}{\color[rgb]{0,0,0}.}%
}}}}
\put(3451,-3691){\makebox(0,0)[b]{\smash{{\SetFigFont{20}{24.0}{\rmdefault}{\mddefault}{\updefault}{\color[rgb]{0,0,0}$3$}%
}}}}
\put(3451,-3241){\makebox(0,0)[b]{\smash{{\SetFigFont{20}{24.0}{\rmdefault}{\mddefault}{\updefault}{\color[rgb]{0,0,0}$1$}%
}}}}
\put(4351,-3241){\makebox(0,0)[b]{\smash{{\SetFigFont{20}{24.0}{\rmdefault}{\mddefault}{\updefault}{\color[rgb]{0,0,0}$2$}%
}}}}
\put(6451,-3241){\makebox(0,0)[b]{\smash{{\SetFigFont{20}{24.0}{\rmdefault}{\mddefault}{\updefault}{\color[rgb]{0,0,0}$1$}%
}}}}
\put(7351,-3241){\makebox(0,0)[b]{\smash{{\SetFigFont{20}{24.0}{\rmdefault}{\mddefault}{\updefault}{\color[rgb]{0,0,0}$3$}%
}}}}
\put(6451,-3691){\makebox(0,0)[b]{\smash{{\SetFigFont{20}{24.0}{\rmdefault}{\mddefault}{\updefault}{\color[rgb]{0,0,0}$2$}%
}}}}
\put(9226,-3466){\makebox(0,0)[b]{\smash{{\SetFigFont{20}{24.0}{\rmdefault}{\mddefault}{\updefault}{\color[rgb]{0,0,0}$1$}%
}}}}
\put(10351,-3241){\makebox(0,0)[b]{\smash{{\SetFigFont{20}{24.0}{\rmdefault}{\mddefault}{\updefault}{\color[rgb]{0,0,0}$3$}%
}}}}
\put(9676,-3466){\makebox(0,0)[b]{\smash{{\SetFigFont{20}{24.0}{\rmdefault}{\mddefault}{\updefault}{\color[rgb]{0,0,0}$2$}%
}}}}
\put(4951,-3511){\makebox(0,0)[lb]{\smash{{\SetFigFont{29}{34.8}{\rmdefault}{\mddefault}{\updefault}{\color[rgb]{0,0,0},}%
}}}}
\put(7951,-3511){\makebox(0,0)[lb]{\smash{{\SetFigFont{29}{34.8}{\rmdefault}{\mddefault}{\updefault}{\color[rgb]{0,0,0},}%
}}}}
\end{picture}%
}
\]
If $|\lambda|$ is odd, the definition is the same except that the single corner box forms a one-cell ``monomino"; for example, the three domino tableaux of shape $\langle 3, 3, 1\rangle$ are
\[
\resizebox{.5\textwidth}{!}{\begin{picture}(0,0)%
\includegraphics{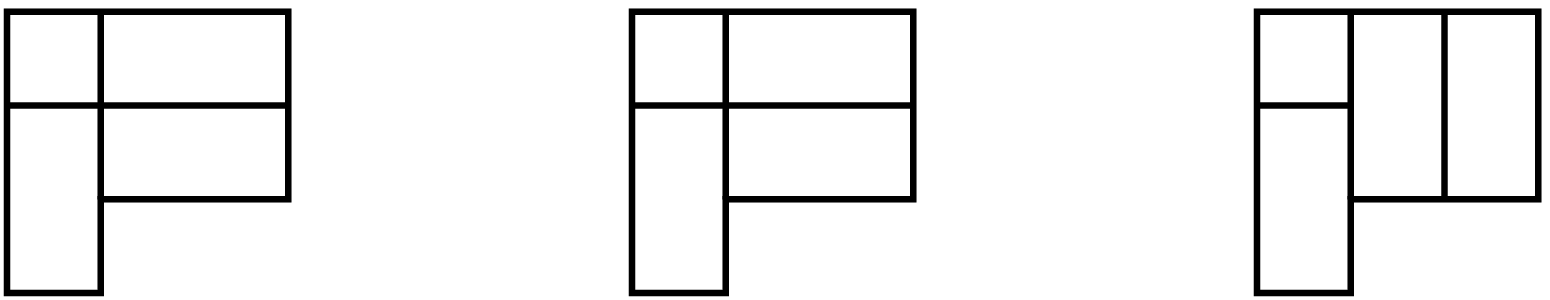}%
\end{picture}%
\setlength{\unitlength}{3947sp}%
\begingroup\makeatletter\ifx\SetFigFont\undefined%
\gdef\SetFigFont#1#2#3#4#5{%
  \reset@font\fontsize{#1}{#2pt}%
  \fontfamily{#3}\fontseries{#4}\fontshape{#5}%
  \selectfont}%
\fi\endgroup%
\begin{picture}(7640,1416)(2968,-5644)
\put(4501,-5311){\makebox(0,0)[lb]{\smash{{\SetFigFont{29}{34.8}{\rmdefault}{\mddefault}{\updefault}{\color[rgb]{0,0,0},}%
}}}}
\put(7501,-5311){\makebox(0,0)[lb]{\smash{{\SetFigFont{29}{34.8}{\rmdefault}{\mddefault}{\updefault}{\color[rgb]{0,0,0},}%
}}}}
\put(10501,-5311){\makebox(0,0)[lb]{\smash{{\SetFigFont{29}{34.8}{\rmdefault}{\mddefault}{\updefault}{\color[rgb]{0,0,0}.}%
}}}}
\put(3901,-5041){\makebox(0,0)[b]{\smash{{\SetFigFont{20}{24.0}{\rmdefault}{\mddefault}{\updefault}{\color[rgb]{0,0,0}$3$}%
}}}}
\put(6901,-5041){\makebox(0,0)[b]{\smash{{\SetFigFont{20}{24.0}{\rmdefault}{\mddefault}{\updefault}{\color[rgb]{0,0,0}$3$}%
}}}}
\put(9676,-4816){\makebox(0,0)[b]{\smash{{\SetFigFont{20}{24.0}{\rmdefault}{\mddefault}{\updefault}{\color[rgb]{0,0,0}$2$}%
}}}}
\put(10126,-4816){\makebox(0,0)[b]{\smash{{\SetFigFont{20}{24.0}{\rmdefault}{\mddefault}{\updefault}{\color[rgb]{0,0,0}$3$}%
}}}}
\put(6901,-4591){\makebox(0,0)[b]{\smash{{\SetFigFont{20}{24.0}{\rmdefault}{\mddefault}{\updefault}{\color[rgb]{0,0,0}$2$}%
}}}}
\put(6226,-5266){\makebox(0,0)[b]{\smash{{\SetFigFont{20}{24.0}{\rmdefault}{\mddefault}{\updefault}{\color[rgb]{0,0,0}$1$}%
}}}}
\put(3901,-4591){\makebox(0,0)[b]{\smash{{\SetFigFont{20}{24.0}{\rmdefault}{\mddefault}{\updefault}{\color[rgb]{0,0,0}$1$}%
}}}}
\put(6226,-4591){\makebox(0,0)[b]{\smash{{\SetFigFont{20}{24.0}{\rmdefault}{\mddefault}{\updefault}{\color[rgb]{0,0,0}$0$}%
}}}}
\put(9226,-5266){\makebox(0,0)[b]{\smash{{\SetFigFont{20}{24.0}{\rmdefault}{\mddefault}{\updefault}{\color[rgb]{0,0,0}$1$}%
}}}}
\put(9226,-4591){\makebox(0,0)[b]{\smash{{\SetFigFont{20}{24.0}{\rmdefault}{\mddefault}{\updefault}{\color[rgb]{0,0,0}$0$}%
}}}}
\put(3226,-5266){\makebox(0,0)[b]{\smash{{\SetFigFont{20}{24.0}{\rmdefault}{\mddefault}{\updefault}{\color[rgb]{0,0,0}$2$}%
}}}}
\put(3226,-4591){\makebox(0,0)[b]{\smash{{\SetFigFont{20}{24.0}{\rmdefault}{\mddefault}{\updefault}{\color[rgb]{0,0,0}$0$}%
}}}}
\end{picture}%
}
\]
The combinatorics of domino tableaux and their connection to representation theory are discussed further in \S\ref{sec:domino}.
  
The number $f^\lambda$ of standard Young tableaux of shape $\lambda$ is given by the \emph{hook-length formula} of Frame--Robinson--Thrall \cite{FRT}:
\[
f^\lambda = \frac{n!}{\prod_{c \in \lambda} h_c}
\]
where $c$ runs over the cells of the Young diagram of $\lambda$ and $h_c$ is the \emph{hook-length} of $c$, i.e., the number of boxes that are in the same row as $c$ and weakly to its right or in the same column and weakly below it.  This number has a natural \emph{$q$-analogue}
\[
f^\lambda(q) := \frac{[n]!_q}{\prod_{c \in \lambda} [h_c]_q},
\]
where for a nonnegative integer $k$ we define $[k]_q := 1 + q + \ldots + q^{k - 1}$ and $[k]!_q := [1]_q \cdot [2]_q \cdots [k]_q$, so that $f^\lambda(1) = f^\lambda$.  It is not clear from this definition, but in fact $f^\lambda(q)$ is a polynomial in $q$ whose coefficients are positive integers.  (It is the generating function for tableaux by a statistic called \emph{comaj} \cite{Krattenthaler}; see also \cite[p.~243]{Macdonald}.)

Let $\chi^\lambda$ denote the irreducible character of the symmetric group indexed by $\lambda$.  We will denote by $\chi^\lambda_\mu$ the result of evaluating $\chi^\lambda$ on a permutation of cycle type $\mu$.  Let $\rho_2(n)$ be the cycle type of $w_0$ in $\Symm_n$, that is, $\rho_2(n)$ is the partition $\langle 2^{n/2} \rangle$ if $n$ is even and $\langle 2^{(n - 1)/2}, 1 \rangle$ if $n$ is odd.  Finally, for a partition $\lambda$, let $b(\lambda) = \sum_i (i-1) \lambda_i$. 

The following theorem shows how these objects are bound together with fixed points of evacuation.

\begin{theorem}[Stembridge {\cite[Thm.~4.3]{Stembridge} and \cite[Thm.~3.1]{StanleyEvac}}]
\label{counting self-evacuating tableaux}
For any partition $\lambda$, there is a bijection between self-evacuating standard Young tableaux of shape $\lambda$ and domino tableaux of shape $\lambda$.  Moreover, the number of these tableaux is given by $(-1)^{b(\lambda)} \cdot \chi^\lambda_{\rho_2(n)} = f^\lambda(-1)$.
\end{theorem}

This theorem is an example of the \emph{$q = -1$ phenomenon} (and more generally the \emph{cyclic sieving phenomenon}), whereby a natural enumerating polynomial for a set gives, upon substitution of $-1$ (or a root of unity) for the variable, the number of fixed points of the set under a natural involution (or cyclic action).

\subsection{Self-evacuating tabloids}
In this section, we state our main enumeration theorem: it is an affine analogue of Theorem~\ref{counting self-evacuating tableaux}, giving the enumeration of self-evacuating tabloids of a given shape $\lambda$.  We begin with some background definitions necessary for the statement, following \cite[Ch.~III]{Macdonald}.

\subsubsection{Kostka--Foulkes and Green's polynomials}
\label{sec: KF&G}

Let $\La$ be the ring of symmetric functions in the variables $x = (x_1, x_2, \ldots)$ and let $t$ be an indeterminate. The \emph{Hall--Littlewood functions} $P_\la(x;t)$ form a basis of the polynomial ring $\La[t]$. These functions interpolate between monomial symmetric functions and Schur functions, in the sense that $P_\la(x;1) = m_\la(x)$ and $P_\la(x;0) = s_\la(x)$. The \emph{Kostka--Foulkes polynomials} $K_{\la \mu}(t)$ are defined to be the coefficients of the Schur functions in the Hall--Littlewood basis:
\begin{equation}
\label{eq: K}
s_\la(x) = \sum_\mu K_{\la \mu}(t) P_\mu(x;t).
\end{equation}
Setting $t = 1$, we have $s_\la(x) = \sum_\mu K_{\la \mu}(1) m_\mu(x)$, and so $K_{\la \mu}(1) = |\SSYT(\la,\mu)|$. In fact, it was shown by Lascoux and Sch\"utzenberger \cite{LScharge, SchutzCharge} that $K_{\la \mu}(t)$ is the generating function over $T \in \SSYT(\la, \mu)$ with respect to a statistic $c(T)$ called \emph{charge}.\footnote{There are several ways to compute charge: in addition to the original references mentioned above, see \cite[\S III.6]{Macdonald} and \cite[\S 5.4]{ShimozonoDummies}. The latter reference gives an algorithm for \emph{cocharge} that applies even when the content $\mu$ is not a partition. (For $T \in \SSYT(\la, \mu)$, the cocharge of $T$ is equal to $\sum_{i < j} \min(\mu_i, \mu_j) - c(T)$.)}


Similarly, define the polynomials $X^\mu_\rho(t)$ to be the coefficients of the power sum symmetric functions in the Hall--Littlewood basis:
\begin{equation}
\label{eq: X}
p_\rho(x) = \sum_\mu X^\mu_\rho(t) P_\mu(x;t).
\end{equation}
Setting $t = 0$ gives $p_\rho(x) = \sum_\mu X^\mu_\rho(0) s_\mu(x)$, so $X^\mu_\rho(0) = \chi^\mu_\rho$ is the value of the irreducible symmetric group character $\chi^\mu$ on a permutation of cycle type $\rho$. Furthermore, it follows from \eqref{eq: K} and \eqref{eq: X} that
\begin{equation}
\label{eq: X in terms of K}
X^\mu_\rho(t) = \sum_\la \chi^\la_\rho K_{\la \mu}(t).
\end{equation}
The polynomial $X^\mu_\rho(t)$ has degree $b(\mu)$, and the \emph{Green's polynomial} $\Green^\mu_\rho(q)$ is defined by
\[
\Green^\mu_\rho(q) = q^{b(\mu)} X^\mu_\rho(q^{-1}).
\]
Green's polynomials were originally defined by Green in his study \cite{Green} of the representation theory of the finite general linear groups.

\subsubsection{Statement of the main result}
\label{sec: fixed point statement}

We now state our main enumerative theorem, giving an affine analogue of Theorem~\ref{counting self-evacuating tableaux}.

\begin{theorem}
\label{main counting theorem}
Suppose that $\lambda = \langle \lambda_1^{m_1}, \lambda_2^{m_2}, \ldots, \lambda_k^{m_k} \rangle$ is a partition of $n$ having $k$ distinct part-sizes.  For $1 \leq i \leq k$, let $\lambda\mathord{\downarrow}^{(\lambda_i)}_{(\lambda_i - 2)}$ be the partition formed by replacing a part of $\lambda$ of size $\lambda_i$ with one of size $\lambda_i - 2$, and let $\lambda\mathord{\downarrow}^{(\lambda_i, \lambda_i)}_{(\lambda_i - 1, \lambda_i - 1)}$ be the partition formed by replacing two parts of $\lambda$ of size $\lambda_i$ with two parts of size $\lambda_i - 1$. Then for $n \geq 2$, the number $t(\lambda)$ of self-evacuating tabloids of shape $\lambda$ satisfies the recurrence relation
\[
t(\lambda) = \sum_{\substack{i \colon \lambda_i \geq 2, \\ m_i \textrm{ is odd}}} t\left(\lambda\mathord{\downarrow}^{(\lambda_i)}_{(\lambda_i - 2)}\right) + \sum_{i=1}^k 2\left\lfloor \frac{m_i}{2} \right\rfloor \cdot t\left(\lambda\mathord{\downarrow}^{(\lambda_i, \lambda_i)}_{(\lambda_i - 1, \lambda_i - 1)}\right).
\]
Moreover, $t(\lambda)$ is given by the evaluation of a Green's polynomial at $q = -1$:
\begin{equation}
\label{eq: Green eval}
t(\lambda) = \Green^\lambda_{\rho_2(n)}(-1).
\end{equation}
\end{theorem}

We make some remarks on this result here; its proof is deferred to the following section.

\begin{remark}
Although we do not have an affine analogue of domino tableaux, it is tempting to view the first sum in the recurrence relation as removing a horizontal domino (i.e., two cells in a single row) from $\lambda$, while the second sum removes a vertical domino (i.e., two cells in adjacent rows).  The recurrence becomes particularly nice when all parts of $\lambda$ are distinct (all multiplicities are $1$), as in the next corollary.
\end{remark}

\begin{corollary}
\label{distinct parts corollary}
Suppose $\lambda = \langle \lambda_1, \ldots, \lambda_k \rangle$ is a partition with distinct parts. For $1 \leq i \leq k$, let $\lambda\mathord{\downarrow}^{(\lambda_i)}_{(\lambda_i - 2)}$ be the partition formed by replacing a part of $\lambda$ of size $\lambda_i$ with one of size $\lambda_i - 2$. Then
\[
t(\lambda) = \sum_{i\colon \lambda_i \geq 2} t\left(\lambda\mathord{\downarrow}^{(\lambda_i)}_{(\lambda_i - 2)}\right).
\]
\end{corollary}

\begin{remark}
It follows easily from the results of \S\ref{sec:R-matrix} that
for any composition $\la$, the number of self-evacuating tabloids of shape $\la$ is equal to $t(\sort(\la))$, where $\sort(\la)$ is the partition obtained by sorting the parts of $\la$ into decreasing order.
\end{remark}

\begin{remark}
The evaluation of the Green's polynomial in \eqref{eq: Green eval} is not an instance of the $q = -1$ phenomenon, in two respects: first, Green's polynomials have both positive and negative coefficients and so do not count tabloids by a statistic.  Second, by \cite[\S III.7, Ex.~6]{Macdonald}, one has that the Green's polynomial evaluation
$\Green^\la_{\langle 1^n \rangle}(1)$ (with lower index partition $\langle 1^n \rangle$) produces the number $\frac{|\la|!}{\la_1! \cdot \la_2! \cdots} = |\mathcal{T}(\la)|$ of tabloids of shape $\la$, while $\Green^\la_{\rho_2(n)}(1)$ is a smaller number.
%
\end{remark}

\subsection{Proof of 
the main theorem}

In this section, we prove our main enumerative result, Theorem~\ref{main counting theorem}.

In \cite[Prop.~10.3]{KimSpringer}, the third-named author gives a recurrence relation for Green's polynomial evaluations at $-1$.  In our notation, his result says that
\[
\Green^\lambda_{\rho_2(n)}(-1) =   
\sum_{i=1}^k 2\left\lfloor \frac{m_i}{2} \right\rfloor \cdot \Green^{\lambda\mathord{\downarrow}^{(\lambda_i, \lambda_i)}_{(\lambda_i - 1, \lambda_i - 1)}}_{\rho_2(n-2)}(-1)
+ 
\sum_{\substack{i \colon \lambda_i \geq 2, \\ m_i \textrm{ is odd}}} \left(\textrm{a sign}\right) \Green^{\lambda\mathord{\downarrow}^{(\lambda_i)}_{(\lambda_i - 2)}}_{\rho_2(n-2)}(-1)
+
\textrm{a third sum}.
\]
In our setting (which involves setting the parameter called $k$ in the reference equal to $1$), the third term is vacuous and the sign in the second term is always $+1$. Thus, the recurrence in Theorem~\ref{main counting theorem} follows from the Green's polynomial evaluation \eqref{eq: Green eval}.

It remains to prove \eqref{eq: Green eval}. We first reduce this result to an evaluation of the Kostka--Foulkes polynomials at $q = -1$. Let $\mu$ be a partition of $n$ with exactly $d$ nonzero parts. From the discussion in \S\ref{sec: KF&G}, we have
\begin{equation*}
\Green^\mu_{\rho_2(n)}(q) = \sum_{\la} \chi^\la_{\rho_2(n)} \tw{K}_{\la \mu}(q),
\end{equation*}
where
\begin{equation}
\label{eq: defn cocharge KF}
\tw{K}_{\la \mu}(q) = \sum_{T \in {\rm SSYT}(\la,\mu)} q^{b(\mu) - c(T)}
\end{equation}
is the \emph{cocharge Kostka--Foulkes polynomial}. Note that the upper index of the Green's polynomial is now $\mu$ rather than $\la$. 

Let $u(\la)$ be the number of standard Young tableaux of shape $\la$ fixed by evacuation, and let $v(\la, \mu)$ be the number of semistandard Young tableaux of shape $\la$ and content $\mu$ that are fixed by the map $e^*_d$ defined in \S \ref{sec: evac and RSK}. Proposition \ref{prop: alt alg for e} implies that
\begin{equation*}
t(\mu) = \sum_{\la} u(\la)v(\la,\mu).
\end{equation*}
By Theorem~\ref{counting self-evacuating tableaux}, we have $u(\lambda) = (-1)^{b(\lambda)} \chi^\lambda_{\rho_2(n)}$.  
Thus, Theorem \ref{main counting theorem} is an immediate consequence of the following result. 



\begin{theorem}
\label{thm: KF at -1}
The number of elements of ${\rm SSYT}(\la,\mu)$ fixed by $e^*_d$ (where $d$ is the number of parts of $\mu$) is given by
\[
v(\la,\mu) = (-1)^{b(\la)} \tw{K}_{\la \mu}(-1).
\]
\end{theorem}

The rest of this section is devoted to the proof of Theorem \ref{thm: KF at -1}. The key tool is Kirillov and Reshetikhin's bijection between semistandard tableaux and rigged configurations, which sheds light on the relationship between the charge of a tableau and its image under $e_d^*$.

\begin{remark}
In contrast to \eqref{eq: Green eval}, Theorem \ref{thm: KF at -1} is an example of the $q=-1$ phenomenon. This result was essentially already known in the case where all part multiplicities of $\mu$ are even. Indeed, using the discussion in \S \ref{sec: evac and RSK}, it is not difficult to deduce that $e_d^*$ commutes with the action of $\Symm_d$. Hence, if $\mu = \langle 1^{2m_1}, 2^{2m_2}, \ldots \rangle$, then the number of fixed points of $e_d^*$ on $\SSYT(\la, \mu)$ is equal to the number of fixed points of $e_d^*$ on $\SSYT(\la, \tw{\mu})$, where $\tw{\mu} = \langle 1^{m_1}, 2^{m_2}, \ldots, 2^{m_2}, 1^{m_1} \rangle$. Since the composition $\tw{\mu}$ is fixed by $w_0$, $e_d^*$ reduces to $e_d$ on $\SSYT(\la, \tw{\mu})$ by Definition-Proposition~\ref{si action}. Stembridge showed that the fixed points of evacuation on $\SSYT(\la, \tw{\mu})$ are in bijection with \emph{semistandard domino tableaux} of content $\langle 1^{m_1}, 2^{m_2}, \ldots \rangle$ \cite[Cor.~4.2]{Stembridge}. By results of Lascoux, Leclerc, and Thibon, the number of such domino tableaux is $(-1)^{b(\la) + b(\mu)}K_{\la \mu}(-1) = (-1)^{b(\la)}\tw{K}_{\la \mu}(-1)$ (see \cite[Thm.~9.17]{KostkaExpository}).
\end{remark}

\subsubsection{Partitions and $q$-binomial coefficients}

 Before discussing rigged configurations, we need a brief digression on $q$-binomial coefficients.  For a partition $\eta$, write $\eta \subset R(a,b)$ if $\eta$ fits inside the $a \times b$ rectangle, i.e., if $\eta$ has at most $a$ rows and $b$ columns. Given $\eta = \langle \eta_1, \ldots, \eta_a \rangle \subset R(a,b)$, let $\overline{\eta} = \langle b-\eta_a,b-\eta_{a-1}, \ldots, b-\eta_1 \rangle$ be the $180^\circ$ rotation of the complement of $\eta$ in $R(a,b)$. We say that $\eta$ is \emph{self-complementary} if $\overline{\eta} = \eta$.  (When we use this term, the values of $a$ and $b$ will be clear from context.)

Define the \emph{$q$-binomial coefficient} $\qbin{n}{k}{q} := \frac{[n]!_q}{[k]!_q \cdot [n-k]!_q}$, where $[a]!_q$ is the $q$-factorial defined in \S \ref{sec: q=-1}. It is well-known (see, e.g., \cite[\S 1.7]{EC1}) that
\begin{equation}
\label{eq: q-binomial}
\qbin{a+b}{a}{q} = \sum_{\eta \subset R(a,b)} q^{|\eta|}.
\end{equation}
The following result is a prototypical example of the $q = -1$ phenomenon -- for example, it is a special case of \cite[Thm.~1.1]{StembridgePlanePartitions}.

\begin{lemma}
\label{lem: q-binomial -1}
The evaluation $\qbin{a+b}{a}{-1}$ of the $q$-binomial coefficient $\qbin{a+b}{a}{q}$ at $q=-1$ is equal to the number of self-complementary partitions in $R(a,b)$.
\end{lemma}
%
%

\subsubsection{Rigged configurations}
\label{sec: rigged configs}

Let $\la = \langle \la_1, \ldots, \la_m \rangle$ and $\mu = \langle \mu_1, \ldots, \mu_d \rangle$ be two partitions of $n$, such that $\mu$ has exactly $d$ nonzero parts. Let $\nu = (\nu^{(1)}, \nu^{(2)}, \ldots)$ be a sequence of partitions. Let $m_r^{(k)}(\nu)$ be the number of parts of size $r$ in $\nu^{(k)}$, and let $\alpha_i^{(k)}(\nu)$ be the length of the $i$th column of $\nu^{(k)}$. Define the \emph{vacancy numbers} of $\nu$ by
\[
P_r^{(k)}(\nu) = \sum_{i \leq r} (\alpha_i^{(k-1)}(\nu) + \alpha_i^{(k+1)}(\nu) - 2\alpha_i^{(k)}(\nu)),
\]
where we take $\alpha_i^{(0)}(\nu)$ to be the length of the $i$th column of $\mu$. We say that $\nu$ is an \emph{admissible configuration} of \emph{type $(\la, \mu)$} if
\begin{equation}
\label{eq: lambda condition}
|\nu^{(k)}| = \la_{k+1} + \ldots + \la_m
\end{equation}
and $P_r^{(k)}(\nu) \geq 0$ for all $r,k$.

A \emph{rigged configuration} of \emph{type $(\la, \mu)$} is a pair $(\nu, J)$, where $\nu$ is an admissible configuration of type $(\la, \mu)$, and $J$ consists of a partition $J_r^{(k)} \subset R\left(m_r^{(k)}(\nu), P_r^{(k)}(\nu)\right)$ for each $r,k$. Equivalently, $J$ consists of an assignment of non-negative integers (called \emph{riggings}) to each row of the partitions $\nu^{(k)}$, such that the integer assigned to a length-$r$ row of $\nu^{(k)}$ is at most $P_r^{(k)}(\nu)$, and among rows of $\nu^{(k)}$ of the same length, the ordering of the riggings is irrelevant. Let ${\rm RC}(\la,\mu)$ denote the set of rigged configurations of type $(\la, \mu)$.

\begin{example}
\label{ex: RC}
Here is a rigged configuration of type $(\langle5,2,2,1\rangle, \langle3,2,2,2,1\rangle)$:
\[
\begin{array}{cccc}
{\tableau[sY]{,, \\ , \\ , \\ , \\ , \bl}} & {\tableau[sY]{\bl 0, , , \bl 2 \\ \bl 1, , , \bl 2 \\ \bl 1,, \bl 1}} & {\tableau[sY]{\bl 0, , , \bl 0 \\ \bl 0, , \bl 0}} & {\tableau[sY]{\bl 0, , \bl 0}} \\
\mu & \nu^{(1)} & \nu^{(2)} & \nu^{(3)}
\end{array}.
\]
The number to the left of a row is its rigging, and the number to the right is its vacancy number. The partition $\mu$ is not part of the rigged configuration, but it is included to facilitate the calculation of the vacancy numbers $P_r^{(1)}(\nu)$. For example, the last row of $\nu^{(1)}$ has length 1, so its vacancy number is $P_1^{(1)}(\nu) = 5+2-6=1$; the first two rows of $\nu^{(1)}$ have length 2, so their vacancy number is $P_2^{(1)}(\nu) = 9+3-10 = 2$.
\end{example}

%
%

Kirillov and Reshetikhin \cite{KirResh} constructed a remarkable bijection between semistandard tableaux of shape $\la$ and content $\mu$ and rigged configurations of type $(\la, \mu)$. This bijection interacts nicely with both the charge statistic and the operation $e^*_d$ from \S \ref{sec: evac and RSK}. Let $\Theta : {\rm RC}(\la,\mu) \rightarrow {\rm RC}(\la,\mu)$ be the map $(\nu, J) \mapsto (\nu, \overline{J})$, where $\overline{J}^{(k)}_r$ is the complement of $J^{(k)}_r$ in the rectangle $R(m_r^{(k)}(\nu), P_r^{(k)}(\nu))$. Equivalently, $\Theta$ replaces each rigging $s$ with its \emph{corigging} $P_r^{(k)} - s$. Define the statistic $cc$ by
\[
cc(\nu, J) = cc(\nu) + \sum_{r,k \geq 1} |J_r^{(k)}| \quad \text{where} \quad cc(\nu) = \sum_{r,k \geq 1} \alpha_r^{(k)} (\alpha_r^{(k)} - \alpha_r^{(k+1)}).
\]
The key to the proof of Theorem \ref{thm: KF at -1} is the following result.

\begin{theorem}
\label{thm: KSS}
There is a bijection
\[
\Phi : {\rm SSYT}(\la,\mu) \rightarrow {\rm RC}(\la,\mu)
\]
such that
\begin{enumerate}
\item $\Theta (\Phi(T)) = \Phi (e^*_d(T))$ and
\item $cc ( \Theta ( \Phi(T))) = b(\mu) - c(T)$.
\end{enumerate}
Here $d$ is the number of parts of $\mu$, $c$ is the charge statistic, and $b(\mu) = \sum_i (i-1)\mu_i$.
\end{theorem}

In the Appendix, we present the algorithm for $\Phi$, and we explain how Theorem \ref{thm: KSS} follows from work of Kirillov, Schilling, and Shimozono \cite{KSS}.

\begin{example}
\label{ex: tableau for RC example}
The rigged configuration $(\nu, J)$ in Example \ref{ex: RC} corresponds under $\Phi$ to the tableau
\[
T = {\tableau[sY]{1,1,1,3,4 \\ 2,2 \\ 3,4 \\ 5}}.
\]
This tableau has charge 4. The rigged configuration $\Theta(\nu, J)$ is obtained by changing the riggings of $\nu^{(1)}$ from $0, 1, 1$ to $2,1,0$, so
\[
cc(\Theta(\nu, J)) = 9 + 3 = 12 = b(\langle3,2,2,2,1\rangle) - 4,
\]
in agreement with Theorem \ref{thm: KSS}(2).
\end{example}

It follows from \eqref{eq: defn cocharge KF}, \eqref{eq: q-binomial}, and Theorem \ref{thm: KSS}(2) that
\begin{equation}
\label{eq: fermionic formula}
\tw{K}_{\la,\mu}(q) = \sum_\nu q^{cc(\nu)} \prod_{r,k \geq 1} \qbin{m_r^{(k)}(\nu) + P_r^{(k)}(\nu)}{m_r^{(k)}(\nu)}{q},
\end{equation}
where the sum is over admissible configurations of type $(\la,\mu)$.
The final ingredient needed to prove Theorem \ref{thm: KF at -1} is a lemma about the parity of $cc(\nu)$.

\begin{lemma}
\label{lem: parity of cc}
Let $\nu$ be an admissible $(\la,\mu)$-configuration. If $m_r^{(k)}(\nu) \cdot P_r^{(k)}(\nu) \equiv 0 \pmod{2}$ for all $r,k$, then
\[
cc(\nu) \equiv b(\la) \pmod{2}.
\]
\end{lemma}

\begin{proof}
In this proof, we use the symbol $\equiv$ to mean ``congruent mod 2'', and we omit the dependence of $m_r^{(k)}, P_r^{(k)}, \alpha_r^{(k)}$ on $\nu$. By definition, we have
\[
cc(\nu) = \sum_{r,k \geq 1} \alpha_r^{(k)}(\alpha_r^{(k)} - \alpha_r^{(k+1)}) \equiv \sum_{r,k \geq 1} \alpha_r^{(k)} + \sum_{r,k \geq 1} \alpha_r^{(k)}\alpha_r^{(k+1)}.
\]
By the definition of $\alpha_r^{(k)}$, $\sum_r \alpha_r^{(k)} = |\nu^{(k)}|$.  By \eqref{eq: lambda condition},
\[
\sum_{r,k \geq 1} \alpha_r^{(k)} = \sum_{k \geq 1} |\nu^{(k)}| = \sum_{i \geq 1} (i-1)\la_i = b(\la),
\]
so it remains to show that $\sum_{r,k \geq 1} \alpha_r^{(k)} \alpha_r^{(k+1)} \equiv 0$ given the assumption on $m_r^{(k)}P_r^{(k)}$.

Since $m_r^{(k)} = \alpha_r^{(k)} - \alpha_{r+1}^{(k)}$, we have
\begin{align*}
\sum_{r \geq 1} m_r^{(k)} P_r^{(k)} &= \sum_{r \geq 1} (\alpha_r^{(k)} - \alpha_{r + 1}^{(k)}) \sum_{i \leq r} (\alpha_i^{(k-1)} + \alpha_i^{(k+1)} - 2\alpha_i^{(k)}) \\
&= \sum_{r \geq 1} \alpha_r^{(k)}(\alpha_r^{(k-1)} + \alpha_r^{(k+1)} - 2\alpha_r^{(k)}) \\
&\equiv \sum_{r \geq 1} \alpha_r^{(k)}(\alpha_r^{(k-1)} + \alpha_r^{(k+1)}).
\end{align*}
This implies that
\[
0 \equiv \sum_{\ell \geq 1} \sum_{r \geq 1} m_r^{(2\ell)}P_r^{(2\ell)} \equiv \sum_{r,k \geq 1} \alpha_r^{(k)}\alpha_r^{(k+1)},
\]
completing the proof.
\end{proof}

\begin{proof}[Proof of Theorem \ref{thm: KF at -1}]
If $a$ and $b$ are both odd, then there are no self-complementary partitions in $R(a,b)$, so $\qbin{a+b}{a}{-1} = 0$ by Lemma \ref{lem: q-binomial -1}. This means that the term in \eqref{eq: fermionic formula} corresponding to $\nu$ vanishes at $q=-1$ unless $m_r^{(k)}(\nu) P_r^{(k)}(\nu)$ is even for all $r,k$, so by Lemma \ref{lem: parity of cc}, we have
\[
\tw{K}_{\la,\mu}(-1) = (-1)^{b(\la)} \sum_{\nu} \prod_{r,k \geq 1} \qbin{m_r^{(k)}(\nu) + P_r^{(k)}(\nu)}{m_r^{(k)}(\nu)}{-1}.
\]
By the definition of $\Theta$ and Lemma \ref{lem: q-binomial -1}, the sum on the right-hand side of this expression is equal to the number of rigged configurations of type $(\la,\mu)$ fixed by $\Theta$.  Finally, by Theorem \ref{thm: KSS}(1), this is also the number of fixed points of $e^*_d$ on ${\rm SSYT}(\la,\mu)$, as claimed.
\end{proof}

\subsection{Combinatorial proofs in special cases}
\label{sec:combinatorial proofs}
The proof of Theorem~\ref{main counting theorem} in the previous section is indirect, and it is natural to seek a bijective proof for the recurrence relation.  So far, we have not been able to give a complete proof.  Nevertheless, we give two special cases that may hold ideas leading to an eventual elementary proof.

\begin{proposition}
\label{prop:count rectangles}
Let $N = mn$.  The number of self-evacuating tabloids of shape $\langle m^n \rangle$ is 
\[
2^{m \cdot \lfloor n/2 \rfloor} \cdot \frac{ \lfloor N / 2 \rfloor !}{(m!)^{\lfloor n / 2 \rfloor} \cdot \left(\lfloor N / 2 \rfloor - m \cdot \lfloor n / 2 \rfloor\right)!}.
\]
\end{proposition}
\begin{proof}
This is a straightforward consequence of Proposition~\ref{rectangles}: the map $x \mapsto N + 1 - x$ arranges the elements of $[\ol{N}]$ into $\lfloor N/2\rfloor$ pairs and possibly one fixed point.  The multinomial coefficient accounts for which pairs appear in which rows of the rectangle, while the power of two accounts for the two ways to place the elements of each pair.
\end{proof}

\begin{prop}
\label{even multiplicities conjecture}
Suppose that $\lambda = \langle \lambda_1^{2m_1}, \ldots \rangle$ is a partition in which all part multiplicities $2m_1, \ldots$ are even, and let $\mu = \langle \mu_1, \ldots, \mu_k \rangle := \langle \lambda_1^{m_1}, \ldots \rangle$.  Then 
\[
t(\lambda) = 2^{|\mu|} \cdot \frac{|\mu|!}{\mu_1! \cdots \mu_k! }.
\]
\end{prop}
\begin{proof}
Given a tabloid $T$ of shape $\lambda$, use the combinatorial $R$-matrix to arrange it into the symmetric composition shape $\langle \lambda_1^{m_1}, \lambda_2^{m_2}, \ldots, \lambda_2^{m_2}, \lambda_1^{m_1}\rangle$.  Since the shape is fixed by the reverse permutation $w_0$, we have by Proposition~\ref{prop:R braid} that the last step in the operation ``turn upside-down, subtract values from $n + 1$, and use the combinatorial $R$-matrix to rectify to the original shape'' (as in Theorem~\ref{thm: alg for e}) has no effect.
Consequently, $T$ is fixed by evacuation if and only if its symmetric counterpart satisfies the symmetry that $\ol{i}$ is in row $j$ if and only if $\ol{n + 1 - i}$ is in row $\ell(\lambda) + 1 - j$.  Counting tabloids with this symmetry is straightforward: the pairs $\{i, |\lambda| + 1 - i\}$ may be assigned to symmetric pairs of rows in $ \frac{|\mu|!}{\mu_1! \cdots \mu_k! }$ ways, and each pair may independently choose which element is in the higher row and which is in the lower.
\end{proof}

\begin{remark}
Exactly the same technique may be used to find a (nearly as nice) formula for the case when exactly one of the multiplicities is odd.
\end{remark}

\begin{remark}
The counting formula in Proposition~\ref{even multiplicities conjecture} establishes the recurrence relation in Theorem~\ref{main counting theorem} for shapes with all multiplicities even: the first summand is empty, and the shapes $\lambda\downarrow^{(\lambda_i, \lambda_i)}_{(\lambda_i - 1, \lambda_i - 1)}$ still have all multiplicities even, and so (making the substitution $m_i \mapsto 2m_i$ and defining $\mu$ as in the statement of Proposition~\ref{even multiplicities conjecture}) one is left checking the identity
\[
2^{|\mu|} \frac{|\mu|!}{\mu_1!  \cdots \mu_k!} = 2^{|\mu| - 1} \cdot \left( 2 \frac{(|\mu| - 1)!}{(\mu_1 - 1)! \cdots \mu_k!} + \ldots + 2 \frac{(|\mu| - 1)!}{\mu_1! \cdots (\mu_k - 1)!}\right),
\]
which is just the multinomial version of Pascal's recurrence.
\end{remark}

\section{Open problems and remarks}
\label{sec:open problems}

\subsection{Direct combinatorial interpretation of recurrence}
\label{sec:domino}

The major enumerative problem suggested by our work is to give a direct bijective proof of the recurrence relation in Theorem~\ref{main counting theorem}.  The basic approach used in the proof of Proposition~\ref{even multiplicities conjecture}, using the combinatorial $R$-matrix to move pairs of equal-length rows to symmetric positions around the center, also yields a reduction of Theorem~\ref{main counting theorem} to the special case Corollary~\ref{distinct parts corollary} of partitions with distinct parts.  In particular, a combinatorial proof of the corollary, paired with this argument, would yield a fully combinatorial proof of the main theorem.

Unfortunately, we have been able to give a direct combinatorial proof of Corollary~\ref{distinct parts corollary} only in the cases that $\ell(\lambda) = 1$ (trivial) or $2$.  The latter argument, establishing that $t(a, b) = t(a - 2, b) + t(a, b - 2)$ via direct bijections when $a > b > 1$, consists of several cases, using the characterization of evacuation in terms of the combinatorial $R$-matrix (Theorem~\ref{thm: alg for e}) and the combinatorics of matchings (equivalently, Dyck words).  We have been unsuccessful in extending this argument to partitions with more than $2$ distinct parts.

As an alternative approach, one might ask for a definition of ``domino tabloids,'' whose enumeration satisfies the recurrence and that can be put in bijection with the fixed points of evacuation.  One might also expect these (hypothetical) domino tabloids to play a role in combinatorics or representation theory similar to the role played by domino tableaux.  For example, domino tableaux of special shapes index Kazhdan--Lusztig cells in the Weyl groups of types $B_n$ and $C_n$ \cite{BV, Garfinkle3}; they play a role in the classification of primitive ideals of enveloping algebras of complex semisimple Lie algebras \cite{Garfinkle1, Garfinkle2, Garfinkle3}; they are associated with insertion procedures having various applications to permutations and signed permutations \cite{Dennises, Garfinkle1, Garfinkle2, Garfinkle3}; and they are related to fixed points of evacuation by a beautiful bijection involving growth diagrams \cite{vanL}.

\subsection{A dihedral group's worth of actions}
In the finite symmetric group, the Dynkin diagram is a path, having a unique nontrivial automorphism.  In the affine symmetric group, the Dynkin diagram is a cycle, and so has a full dihedral group's worth of symmetries, of which the operation $r$ is just one.  We discuss these is more detail below.

The rotation symmetry of the diagram sending $s_i \mapsto s_{i + 1}$ (taking indices modulo $n$) corresponds in the group to conjugation by the ``shift permutation" $s = [2, 3, 4, \ldots, n + 1]$; in terms of the permutation matrix, it simultaneously shifts the axes one unit to the left and one unit up.  (The other rotation symmetries are just powers of this symmetry.)  After applying AMBC, the effect of this map is to add $1$ to each entry of the $P$ and $Q$ tabloids \cite[\S6]{CLP}; that is, it corresponds exactly to the promotion map of \S\ref{sec:R-matrix}.  This map is one of the original examples of the cyclic sieving phenomenon, using the obvious $q$-multinomial coefficient; see \cite[Prop.~4.4]{RSW} (where one should interpret a tabloid as a flag of sets by letting the $k$th member of the flag be the union of the first $k$ rows of the tabloid).

When $n$ is odd, the $n$ reflection symmetries of the Dynkin diagram are all conjugate under the action of $s$; in terms of the permutation matrix, they are each given by a rotation fixing some diagonal cell of $\ZZ \times \ZZ$.  Passing through AMBC, we recover $n$ involutions on tabloids, of which the affine evacuation map $e$ is just one.  In terms of the dual equivalence graph (Theorem~\ref{thm: def1 for evacuation}), these variant evacuations permute the RRSS tabloids differently, sending $\RRSS(\lambda, \ol{k})$ to $\RRSS(\lambda, \ol{n - k + c})$ for some $c$, and similarly sending Knuth moves of type $\ol{k}$ to those of type $\ol{n - 1 + c - k}$.  (The map $e$ corresponds to $c = 0$.) One can compute this variant evacuation by reversing the order of the rows, replacing each entry $\ol{i}$ of the tabloid with $\ol{n - i + 1 + c}$, and then applying the combinatorial $R$-matrix $R_{w_0}$ to recover the original shape (compare Theorem~\ref{thm: alg for e}).  Since these maps are conjugate to $e$ in the automorphism group of the dual equivalence graph, they have the same fixed-point structure; in particular, Theorem~\ref{main counting theorem} is valid for all variants.

When $n$ is even, the situation is different: the $n$ reflection symmetries of the Dynkin diagram come in two conjugacy classes under the action of $s$, one of which consists of rotations of the permutation matrix that fix a diagonal cell, and the other of which consists of rotations fixing no cells.  As in the case $n$ odd, these maps are all automorphisms of the dual equivalence graph having the same form as those in the previous paragraph, but the automorphisms with $c$ even are not conjugate to those with $c$ odd.
\begin{example}
There are six tabloids of shape $\langle 2, 2 \rangle$, with dual equivalence graph given by the disjoint union of two paths:
\[
\resizebox{.9\textwidth}{!}{\begin{picture}(0,0)%
\includegraphics{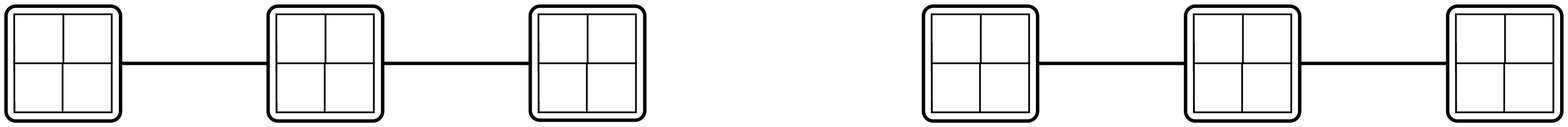}%
\end{picture}%
\setlength{\unitlength}{3947sp}%
\begingroup\makeatletter\ifx\SetFigFont\undefined%
\gdef\SetFigFont#1#2#3#4#5{%
  \reset@font\fontsize{#1}{#2pt}%
  \fontfamily{#3}\fontseries{#4}\fontshape{#5}%
  \selectfont}%
\fi\endgroup%
\begin{picture}(14311,1121)(1228,-3714)
\put(15211,-3481){\makebox(0,0)[b]{\smash{{\SetFigFont{20}{24.0}{\rmdefault}{\mddefault}{\updefault}{\color[rgb]{0,0,0}$\ol{4}$}%
}}}}
\put(1561,-3481){\makebox(0,0)[b]{\smash{{\SetFigFont{20}{24.0}{\rmdefault}{\mddefault}{\updefault}{\color[rgb]{0,0,0}$\ol{3}$}%
}}}}
\put(2011,-3031){\makebox(0,0)[b]{\smash{{\SetFigFont{20}{24.0}{\rmdefault}{\mddefault}{\updefault}{\color[rgb]{0,0,0}$\ol{2}$}%
}}}}
\put(2011,-3481){\makebox(0,0)[b]{\smash{{\SetFigFont{20}{24.0}{\rmdefault}{\mddefault}{\updefault}{\color[rgb]{0,0,0}$\ol{4}$}%
}}}}
\put(3961,-3031){\makebox(0,0)[b]{\smash{{\SetFigFont{20}{24.0}{\rmdefault}{\mddefault}{\updefault}{\color[rgb]{0,0,0}$\ol{1}$}%
}}}}
\put(3961,-3481){\makebox(0,0)[b]{\smash{{\SetFigFont{20}{24.0}{\rmdefault}{\mddefault}{\updefault}{\color[rgb]{0,0,0}$\ol{2}$}%
}}}}
\put(4411,-3031){\makebox(0,0)[b]{\smash{{\SetFigFont{20}{24.0}{\rmdefault}{\mddefault}{\updefault}{\color[rgb]{0,0,0}$\ol{3}$}%
}}}}
\put(4411,-3481){\makebox(0,0)[b]{\smash{{\SetFigFont{20}{24.0}{\rmdefault}{\mddefault}{\updefault}{\color[rgb]{0,0,0}$\ol{4}$}%
}}}}
\put(6361,-3031){\makebox(0,0)[b]{\smash{{\SetFigFont{20}{24.0}{\rmdefault}{\mddefault}{\updefault}{\color[rgb]{0,0,0}$\ol{3}$}%
}}}}
\put(6361,-3481){\makebox(0,0)[b]{\smash{{\SetFigFont{20}{24.0}{\rmdefault}{\mddefault}{\updefault}{\color[rgb]{0,0,0}$\ol{1}$}%
}}}}
\put(6811,-3031){\makebox(0,0)[b]{\smash{{\SetFigFont{20}{24.0}{\rmdefault}{\mddefault}{\updefault}{\color[rgb]{0,0,0}$\ol{4}$}%
}}}}
\put(6811,-3481){\makebox(0,0)[b]{\smash{{\SetFigFont{20}{24.0}{\rmdefault}{\mddefault}{\updefault}{\color[rgb]{0,0,0}$\ol{2}$}%
}}}}
\put(9961,-3031){\makebox(0,0)[b]{\smash{{\SetFigFont{20}{24.0}{\rmdefault}{\mddefault}{\updefault}{\color[rgb]{0,0,0}$\ol{1}$}%
}}}}
\put(9961,-3481){\makebox(0,0)[b]{\smash{{\SetFigFont{20}{24.0}{\rmdefault}{\mddefault}{\updefault}{\color[rgb]{0,0,0}$\ol{2}$}%
}}}}
\put(10411,-3031){\makebox(0,0)[b]{\smash{{\SetFigFont{20}{24.0}{\rmdefault}{\mddefault}{\updefault}{\color[rgb]{0,0,0}$\ol{4}$}%
}}}}
\put(10411,-3481){\makebox(0,0)[b]{\smash{{\SetFigFont{20}{24.0}{\rmdefault}{\mddefault}{\updefault}{\color[rgb]{0,0,0}$\ol{3}$}%
}}}}
\put(12361,-3031){\makebox(0,0)[b]{\smash{{\SetFigFont{20}{24.0}{\rmdefault}{\mddefault}{\updefault}{\color[rgb]{0,0,0}$\ol{2}$}%
}}}}
\put(12361,-3481){\makebox(0,0)[b]{\smash{{\SetFigFont{20}{24.0}{\rmdefault}{\mddefault}{\updefault}{\color[rgb]{0,0,0}$\ol{1}$}%
}}}}
\put(12811,-3031){\makebox(0,0)[b]{\smash{{\SetFigFont{20}{24.0}{\rmdefault}{\mddefault}{\updefault}{\color[rgb]{0,0,0}$\ol{4}$}%
}}}}
\put(12811,-3481){\makebox(0,0)[b]{\smash{{\SetFigFont{20}{24.0}{\rmdefault}{\mddefault}{\updefault}{\color[rgb]{0,0,0}$\ol{3}$}%
}}}}
\put(14761,-3031){\makebox(0,0)[b]{\smash{{\SetFigFont{20}{24.0}{\rmdefault}{\mddefault}{\updefault}{\color[rgb]{0,0,0}$\ol{2}$}%
}}}}
\put(14761,-3481){\makebox(0,0)[b]{\smash{{\SetFigFont{20}{24.0}{\rmdefault}{\mddefault}{\updefault}{\color[rgb]{0,0,0}$\ol{1}$}%
}}}}
\put(15211,-3031){\makebox(0,0)[b]{\smash{{\SetFigFont{20}{24.0}{\rmdefault}{\mddefault}{\updefault}{\color[rgb]{0,0,0}$\ol{3}$}%
}}}}
\put(1561,-3031){\makebox(0,0)[b]{\smash{{\SetFigFont{20}{24.0}{\rmdefault}{\mddefault}{\updefault}{\color[rgb]{0,0,0}$\ol{1}$}%
}}}}
\end{picture}%
}
\] 
The action of the usual evacuation map is described in Proposition~\ref{rectangles}; in terms of the graph, this fixes the component containing $\tableau[sY]{\ol{1}, \ol{2} \\ \ol{3}, \ol{4}}$ pointwise and swaps the end vertices in the other component, giving a total of $4$ fixed points (in agreement with Proposition~\ref{prop:count rectangles}).  However, members of the other conjugacy class of evacuation maps (with $c$ odd) actually swap the two components, and so have $0$ fixed points.
\end{example}

In the case of standard Young tableaux of any fixed shape, promotion and evacuation together give an action of a dihedral group \cite[Lem.~9]{SchutzenbergerPromotion}, but the order of promotion is not well understood in general.  However, when the underlying shape $\lambda$ is a rectangle, the order is known, and promotion exhibits a cyclic sieving phenomenon with the $q$-analogue of the hook-length formula \cite[Thm.~1.3]{Rhoades}.  In the case that the rectangle has even size $n$, there is again a second conjugacy class of reflections in the dihedral group.  The number of fixed points of these evacuation variants is equal to $\pm \chi^\lambda_{\rho'_2(n)}$, where $\rho'_2(n) = \langle 2^{n/2 - 1}, 1, 1\rangle$ is the cycle type of the element $w_0 c_n$ and $c_n$ is the long cycle $2 3 \cdots n 1$ \cite[Prop.~7.3]{Rhoades}.\footnote{In the case of semistandard tableaux of shape $\lambda$, with entries at most $d$, the fixed points of usual evacuation ($e_d$) are given by substituting the alternating sequence $(1, -1, 1, -1, 1, -1, \ldots, \pm 1)$ into the Schur polynomial $s_\lambda(x_1, \ldots, x_d)$ \cite[Thm.~3.1]{Stembridge}. When $d$ is even and $\lambda$ is a rectangle, the fixed points of the other conjugacy class of reflections are given by substituting the almost alternating sequence $\pm (1, 1, -1, 1, -1, 1, \ldots, \pm 1)$ into the same polynomial \cite[Thm.~7.6]{Rhoades}.}  In our setting, we have the following conjecture.

\begin{conjecture}
Suppose that $n$ is even, and let $e' = e \circ \pr$ be one of the alternate evacuation maps described above.  For $\lambda \vdash n$, let $t'(\lambda)$ denote the number of tabloids of shape $\lambda$ fixed by $e'$.  Then $t'(\lambda)$ is given by the evaluation of a Green's polynomial at $q = -1$:
\[
t'(\lambda) = \Green^\lambda_{\rho'_2(n)}(-1).
\]

Moreover, one has that $t'(\lambda)$ satisfies the same recurrence as $t(\lambda)$ (as in Theorem~\ref{main counting theorem}), with initial values $t'(\langle 2 \rangle) = 1$ and $t'(\langle 1, 1\rangle) = 0$ (whereas $t(\langle 1, 1\rangle) = 2$).
\end{conjecture}

\subsection{Weights}
The statement of Theorem~\ref{evacuation exists theorem} raises the following natural question.
\begin{question}
\label{weight question}
If $w \ambcs (P, Q, \rho)$, is there a natural characterization of the weight vector $\rho'$ with the property $r(w) \ambcs (e(P), e(Q), \rho')$?
\end{question}

Very roughly speaking, the entries of $\rho$ track how far the balls (nonzero matrix entries) of $w$ are shifted from the main diagonal.  Since rotation moves every entry above the main diagonal the same distance below the diagonal, to a first approximation we might expect that $\rho' = - \rho$.  (This is essentially what happens when one takes the \emph{inverse} of an affine permutation -- see \cite[Prop.~3.1]{CLP}.)  The following conjecture, based on some numerical evidence, makes this more precise.

\begin{conjecture}
There exists a function $d$ from tabloids of shape $\lambda$ to integer vectors of length $\ell(\lambda)$ such that for every $w$,  $\rho(r(w))$ is the dominant representative (in the sense of \cite[\S5.2]{CPY} or \cite[\S5]{CLP}) of the vector $-\rho(w) + d(Q(w)) - d(P(w))$.
\end{conjecture}

\subsection{More on the asymptotic realization}
\label{sec:asymptotic remark}
One is tempted to conjecture the following stronger version of Proposition~\ref{prop:asymptotic}: for \emph{any} periodic sequence $(T_0, T_1, \ldots)$ of tableaux with limit $T$ (as in \S\ref{sec:asymptotic}), the sequence $(e(T_0), e(T_1), \ldots)$ is periodic with limit $e(T)$.  However, this conjecture is false, as illustrated by the following example.  
\begin{example}
With $n = 2$, consider the following sequence of tableaux:
\def\Tscale{1.3}
\[
\tableau[Y]{1\\2} \; , \;
\tableau[Y]{-1,1,3,5,6,7\\0,2,4} \; , \;
\tableau[Y]{-3,-1,1,3,5,7,8,9,10,11,12\\-2,0,2,4,6} \; , \; \ldots.
\]
The associated limit tabloid is $T = \tableau[sY]{\ol{1} \\ \ol{2}}$.  However, the evacuations of these tableaux are
\[
\tableau[Y]{1\\2} \; , \;
\tableau[Y]{-1,0,1,2,4,6 \\ 3,5,7} \; , \;
\tableau[Y]{-3,-2,-1,0,1,2,3,5,7,9,11 \\ 4,6,8,10,12} \; , \; \ldots
\]
a periodic sequence with limit $\tableau[sY]{\ol{1},  \ol{2}}$ -- not even of the correct shape! \end{example}

On the other hand, one should be able to recover the following version of the conjecture.
\begin{conjecture}
Given a tabloid $T$, let $(T_0, T_1, \ldots)$ be a periodic sequence of tableaux with limit $T$ and \emph{uniformly bounded error}, in the sense that there is an absolute constant $M$ such that for all $k$, all but at most $M$ entries of $T_k$ form an interval in $\ZZ$, and at most $M$ entries of $T_k$ lie in a row other than the limit row.  Then the sequence $(e(T_0), e(T_1), \ldots)$ is periodic with limit $e(T)$.
\end{conjecture}

This conjecture is close in spirit to the question (implicit in \cite{pak}) of whether pairs of ``periodic infinite tableaux" correspond to an ``eventually periodic permutation'' under an appropriate asymptotic version of RS.

\appendix
\section{More on rigged configurations}

In this appendix, we describe the forward direction of the bijection $\Phi : {\rm SSYT}(\la,\mu) \rightarrow {\rm RC}(\la,\mu)$ for the reader's convenience, and we explain how Theorem \ref{thm: KSS} follows from results of Kirillov, Schilling, and Shimozono \cite{KSS}.

\subsection{The rigged configuration bijection}

Recall that a rigged configuration is a pair $(\nu, J)$, where $\nu = (\nu^{(1)}, \nu^{(2)}, \ldots)$ is sequence of partitions, and $J$ is an assignment of a non-negative integer \emph{rigging} to each row of each partition, such that the rigging attached to a row of length $r$ in $\nu^{(k)}$ is at most the vacancy number $P_r^{(k)}(\nu)$. Say that a row of the partition $\nu^{(k)}$ is a \emph{singular string} if its rigging is equal to the corresponding vacancy number; by convention, every partition has a singular string of length $0$.

\begin{defprop}[Kirillov--Reshetikhin \cite{KirResh}]
\label{defprop: RC bijection}
The bijection $\Phi : {\rm SSYT}(\la,\mu) \rightarrow {\rm RC}(\la,\mu)$ of Theorem \ref{thm: KSS} is computed recursively as follows. The empty tableau maps to the empty rigged configuration. Given a non-empty tableau $T$, let $d$ be the maximum entry of $T$, and let $U$ be the tableau obtained from $T$ by removing the right-most box in $T$ that contains $d$. Assume that $\Phi(U) = (\nu, J)$ has been computed; we explain how to construct $\Phi(T) = (\nu', J')$. Suppose the box removed from $T$ is in row $r$. Set $m_r = \infty$ and define a sequence
\[
m_{r-1} \geq m_{r-2} \geq \cdots \geq m_1 \geq 0
\]
by taking $m_j$ to be the biggest number less than or equal to $m_{j+1}$ such that $\nu^{(j)}$ contains a singular string of length $m_j$.  The configuration $\nu'$ is obtained from $(\nu, J)$ by adding one box to a singular string of length $m_j$ in $\nu^{(j)}$ for each $j \leq r-1$. For each of these $r-1$ changed rows, the rigging is chosen so that the row is a singular string with respect to the vacancy numbers of $\nu'$. The riggings of all other rows are unchanged.
\end{defprop}

An algorithm for the other direction is given in \cite[\S 3]{KirResh}.

\begin{example}
Figure \ref{fig: Phi ex} shows one step of the algorithm for $\Phi$. To obtain the rigged configuration in Example \ref{ex: RC} corresponding to the tableau in Example \ref{ex: tableau for RC example}, one performs two more steps of the algorithm. The first of these steps adds a box to the fourth row of $\mu$, which causes the vacancy numbers of $\nu^{(1)}$ to increase from 1 to 2.
\end{example}

\begin{figure}
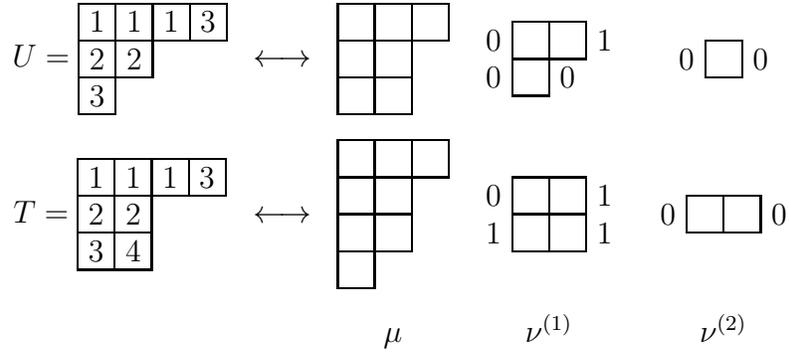


\[
\begin{array}{ccccc}
U = {\tableau[sY]{1,1,1,3 \\ 2,2 \\ 3}} & \longleftrightarrow & {\tableau[sY]{,, \\,\\,}} & {\tableau[sY]{\bl 0, , , \bl 1 \\ \bl 0, , \bl 0}} & {\tableau[sY]{\bl 0, , \bl 0}} \bigskip \\
T = {\tableau[sY]{1,1,1,3 \\ 2,2 \\ 3,4}} & \longleftrightarrow & {\tableau[sY]{,, \\,\\,\\,\bl}} & {\tableau[sY]{\bl 0, , , \bl 1 \\ \bl 1, , , \bl 1}} & {\tableau[sY]{\bl 0, , , \bl 0}} \bigskip \\
&&\mu & \nu^{(1)} & \nu^{(2)}
\end{array}
\]

\caption{One step in the algorithm for $\Phi$. The number to the left of a row is its rigging, and the number to the right its vacancy number.}
\label{fig: Phi ex}
\end{figure}

In \S \ref{sec: rigged configs}, we required that $\mu$ be a partition. In fact, for any strict composition $\mu = \langle \mu_1, \ldots, \mu_d \rangle$, the algorithm in Definition-Proposition \ref{defprop: RC bijection} gives a bijection $\Phi_\mu : \SSYT(\la, \mu) \rightarrow \RC(\la, \sort(\mu))$, where $\sort(\mu)$ is the partition obtained by sorting the parts of $\mu$.\footnote{The role of $\mu$ in the definition of $\RC(\la,\mu)$ (and in the algorithm for $\Phi$) is in the computation of the vacancy numbers $P_r^{(1)}(\nu)$, and only the lengths of the columns of $\mu$ are relevant to this computation. Thus, permuting rows of $\mu$ does not affect these definitions.} The existence of these bijections implies that for a permutation $w \in \Symm_d$, the composition $\Phi_{w \mu}^{-1} \circ \Phi_\mu$ defines a shape-preserving, content-permuting action of the symmetric group on semistandard tableaux. It is asserted in \cite{KSS} that this action agrees with the action introduced in \S \ref{sec: evac and RSK}; we supply the omitted proof.

\begin{lemma}
\label{lem: Phi is invariant}
For $T \in \SSYT(\la,\mu)$ and $w \in \Symm_d$,
\[
\Phi_{w \mu}^{-1} (\Phi_\mu(T)) = w(T),
\]
where $w(T)$ denotes the action of Definition-Proposition \ref{si action}.
\end{lemma}

\begin{proof}
It suffices to consider the generators $s_i$. By \cite[Def.-Prop.~8.1 and Lem.~8.5]{KSS}, the maps $\Psi_i := \Phi_{s_i \mu}^{-1} \circ \Phi_\mu$ are characterized by the properties that $\Psi_i$ does not change entries greater than $i+1$, and $\Psi_{d-1} \circ e_d = e_d \circ \Psi_1$, where $e_d$ is the evacuation map on semistandard tableaux with entries at most $d$. The maps $s_i$ of Definition-Proposition \ref{si action} satisfy the first property by definition, and the second property by, e.g., \cite[Prop.~2.87(iii)]{ShimozonoDummies}.
\end{proof}

\subsection{Relation with work of Kirillov--Schilling--Shimozono}

The rigged configurations discussed in \S \ref{sec: rigged configs} are a special case of a more general theory of rigged configurations developed by Kirillov, Schilling, and Shimozono in \cite{KSS}. In this section we introduce the more general setup and notation of that paper, state several of their main theorems, and explain how these theorems imply Theorem \ref{thm: KSS}.

Let $\mathcal{R} = (\mathcal{R}_1, \ldots, \mathcal{R}_d)$ be a sequence of rectangular partitions, and $\la$ an arbitrary partition. In \cite{KSS}, the authors define a set $\CLR(\la; \mathcal{R})$ of \emph{Littlewood--Richardson tableaux}. The elements of $\CLR(\la; \mathcal{R})$ are a subset of the standard tableaux of shape $\la$, and the cardinality of $\CLR(\la; \mathcal{R})$ is equal to the multiplicity of the Schur function $s_{\la}$ in the product $s_{\mathcal{R}_1} \cdots s_{\mathcal{R}_d}$. In particular, when $\mathcal{R} = \mathcal{R}(\mu) := (\langle \mu_1 \rangle, \ldots, \langle \mu_d \rangle)$ is a sequence of one-row partitions, the cardinality of $\CLR(\la; \mathcal{R}(\mu))$ is the Kostka number $K_{\la \mu}$, and $\CLR(\la; \mathcal{R}(\mu))$ consists of the standardizations\footnote{The \emph{standardization} of a semistandard tableau $T$ of content $\langle \mu_1, \ldots, \mu_d \rangle$ is the standard tableau obtained by replacing the 1's in $T$ with $1, \ldots, \mu_1$, from left to right; the 2's in $T$ with $\mu_1 + 1, \ldots, \mu_1 + \mu_2$, from left to right; etc. This is unrelated to the standardization of a tabloid discussed in \S \ref{sec:special cases}!} of the elements of $\SSYT(\la, \mu)$. Let $\std$ denote the standardization map $\SSYT(\la, \mu) \rightarrow \CLR(\la; \mathcal{R}(\mu))$.

We recall several properties of Littlewood--Richardson tableaux. First, evacuation of standard tableaux restricts to a bijection $e: \CLR(\la; \mathcal{R}) \rightarrow \CLR(\la; w_0(\mathcal{R}))$, where $w_0(\mathcal{R}) = (\mathcal{R}_d, \ldots, \mathcal{R}_1)$ \cite[paragraph after Lem. 3.8]{KSS}. Let $\eta^t$ denote the transpose (or conjugate) of the partition $\eta$, and let $\mathcal{R}^t = (\mathcal{R}_1^t, \ldots, \mathcal{R}_d^t)$. There is a bijection $\trLR : \CLR(\la; \mathcal{R}) \rightarrow \CLR(\la^t; \mathcal{R}^t)$ \cite[Def.~2.7]{KSS}, and a \emph{generalized charge} statistic $c_\mathcal{R} : \CLR(\la; \mathcal{R}) \rightarrow \mathbb{Z}_{\geq 0}$ defined in \cite{ShimozonoLR, SchillWar}. By \cite[Prop.~26 and Thm.~28]{ShimozonoAtoms} or \cite[Lem.~6.5]{SchillWar}, one has
\begin{equation}
\label{eq: charge tr}
c_{\mathcal{R}^t}(\trLR(T)) = n(\mathcal{R}) - c_\mathcal{R}(T)
\end{equation}
for $T \in \CLR(\la; \mathcal{R})$, where $n(\mathcal{R}) = \sum_{i < j} |\mathcal{R}_i \cap \mathcal{R}_j|$  and $\mathcal{R}_1 \cap \mathcal{R}_2$ denotes intersection of Young diagrams. When $\mathcal{R} = \mathcal{R}(\mu)$, $\trLR$ is simply the reflection of a standard tableau over the main diagonal, and
\begin{equation}
\label{eq: LS charge}
c_{\mathcal{R}(\mu)}(T) = c(\std^{-1}(T)),
\end{equation}
where $c$ is the usual charge statistic on semistandard tableaux.\footnote{When $\mathcal{R} = \mathcal{R}(\mu)$, the definition of $c_\mathcal{R}$ in \cite[\S 2.5]{ShimozonoLR} becomes the formula for charge given by Lascoux--Leclerc--Thibon \cite[Thm.~5.1]{LLTcharge}.} Note also that $n(\mathcal{R}(\mu)) = b(\sort(\mu)) = \sum (i-1)(\sort(\mu))_i$.

Kirillov, Schilling, and Shimozono also define a notion of \emph{rigged configuration of type $(\la; \mathcal{R})$}; they denote the set of these by $\RC(\la; \mathcal{R})$. A rigged configuration of type $(\la; \mathcal{R}(\mu))$ is the same as a rigged configuration of type $(\la, \sort(\mu))$, as defined in \S \ref{sec: rigged configs}. The involution $\Theta$ and the statistic $cc$ naturally generalize to rigged configurations of type $(\la; \mathcal{R})$. In \cite[\S 4.1]{KSS}, the authors define a bijection
\[
\ov{\phi}_\mathcal{R} : \CLR(\la; \mathcal{R}) \rightarrow \RC(\la^t; \mathcal{R}^t).
\]
Comparing the algorithm of Definition-Proposition \ref{defprop: RC bijection} with the description of $\ov{\phi}_\mathcal{R}$ in \cite[\S 4.2]{KSS}, one sees that
\[
\Phi_\mu = \ov{\phi}_{\mathcal{R}(\mu)^t} \circ \trLR \circ \std.
\]

The definition of a rigged configuration of type $(\la; \mathcal{R})$ does not depend on the ordering of the rectangles in $\mathcal{R}$, so we identify the sets $\RC(\la; \mathcal{R})$ and $\RC(\la; w_0(\mathcal{R}))$. Also set $\tw{\phi}_\mathcal{R} = \Theta \circ \ov{\phi}_\mathcal{R}$.

\begin{theorem}[{\cite[Thm.~5.6, 9.1]{KSS}}]
\label{thm: KSS actual}
For $T \in \CLR(\la; \mathcal{R})$, one has
\begin{enumerate}
\item $\Theta(\ov{\phi}_\mathcal{R}(T)) = \ov{\phi}_{w_0(\mathcal{R})}(e(T))$ and
\item $cc (\tw{\phi}_\mathcal{R}(T)) = c_\mathcal{R}(T)$.
\end{enumerate}
\end{theorem}

\begin{proof}[Proof of Theorem \ref{thm: KSS}]
Suppose $T \in \SSYT(\la,\mu)$. Recall that $e^* = w_0 \circ e$. By Lemma \ref{lem: Phi is invariant},
\[
\Phi_\mu(e^*(T)) = \Phi_{w_0(\mu)}(e(T)).
\]
By Theorem \ref{thm: KSS actual}(1) and the fact that evacuation commutes with both standardization and transposition of standard tableaux, we have
\[
\Phi_{w_0(\mu)}(e(T)) = \ov{\phi}_{w_0(\mathcal{R}(\mu)^t)} \circ \trLR \circ \std (e(T)) = \Theta (\Phi_{\mu}(T)).
\]
Thus, $\Phi_\mu(e^*(T)) = \Theta (\Phi_\mu(T))$, proving part (1).

For part (2), use Theorem \ref{thm: KSS actual}(2) and \eqref{eq: charge tr}, \eqref{eq: LS charge} to obtain
\[
cc \circ \Theta \circ \Phi_\mu(T) = cc \circ \tw{\phi}_{\mathcal{R}(\mu)^t} \circ \trLR \circ \std(T) = b(\mu) - c(T).
\qedhere
\]
\end{proof}

\bibliography{AffineEvacuation}
\bibliographystyle{alpha}

\end{document}